\documentclass[10pt,twoside,a4paper]{amsart}

\usepackage[english]{babel}
\usepackage[utf8]{inputenc}

\usepackage{amsfonts,amsthm,amssymb,amsmath,amsthm,latexsym,wasysym,stmaryrd,mathrsfs,dsfont,txfonts}
\usepackage{accents,multirow,rotating,graphicx,bigstrut,lscape,multicol,fancyhdr,enumerate,accents,mathtools,exscale}
\usepackage[dvipsnames]{xcolor}

\usepackage[normalem]{ulem}

\usepackage{pdftricks}
\usepackage{color}
\usepackage[pdftex,all]{xy}

\usepackage{hyperref}
\usepackage{appendix}
\usepackage[justification=centering]{caption}
\usepackage{subcaption}

\theoremstyle{theorem}
\newtheorem {theo}{Theorem}[section]
\newtheorem*{theo*}{Theorem}

\newtheorem {lemme}[theo]{Lemma}
\newtheorem*{lemme*}{Lemma}
\newtheorem {prop}[theo]{Proposition}
\newtheorem*{prop*}{Proposition}
\newtheorem {cor}[theo]{Corollary}
\newtheorem*{cor*}{Corollary}

\newtheorem*{cor_proof*}{Corollary (of the proof)}

\newtheorem*{conjecture*}{Conjecture}
\theoremstyle{definition}
\newtheorem {defi}[theo]{Definition}
\newtheorem*{defi*}{Definition}
\newtheorem {nota}[theo]{Notation}
\newtheorem*{nota*}{Notation}
\theoremstyle{remark}
\newtheorem {remarque}[theo]{Remark}
\newtheorem*{remarque*}{Remark}

\newtheorem*{warning*}{Warning}

\newtheorem*{remarques*}{Remarks}

\newtheorem*{warnings*}{Warnings}

\newtheorem*{convention*}{Convention}

\newtheorem*{exemple*}{Example}

\newtheorem*{exemples*}{Examples}

\newtheorem*{question*}{Question}

\newtheorem*{questions*}{Questions}

\newtheorem*{fact*}{Fact}
\newtheorem*{acknowledgments}{Acknowledgments}

\def\N{{\mathds N}}

\def\R{{\mathds R}}
\def\RR{{\mathcal R}}
\def\RRR{{\textrm R}}

\def\Z{{\mathds Z}}

\def\2Z{{\fract{\Z}{2\Z}}}

\def\e{\varepsilon}

\def\p{\partial}

\newcommand{\fract}[2]{\hbox{\leavevmode
  \kern.1em \raise .25ex \hbox{\the\scriptfont0 $#1$}\kern-.1em }\big/
  {\hbox{\kern-.15em \lower .5ex \hbox{\the\scriptfont0 $#2$}} }}

\newcommand{\saut}{\vspace{\baselineskip}}
\newcommand{\dessin}[2]{\vcenter{\hbox{\includegraphics[height=#1]{#2.pdf}}}}

\def\wGraph{w\textnormal{Graph}}

\DeclareMathOperator{\Tube}{Tube}

\begin{document} 

\title{Welded graphs, Wirtinger groups and knotted punctured spheres} 
\author[B. Audoux]{Benjamin Audoux}
\address{Aix Marseille Univ, CNRS, Centrale Marseille, I2M, Marseille, France}
\email{benjamin.audoux@univ-amu.fr}
\author[J.B. Meilhan]{Jean-Baptiste Meilhan} 
\address{Univ. Grenoble Alpes, CNRS, Institut Fourier, F-38000 Grenoble, France}
\email{jean-baptiste.meilhan@univ-grenoble-alpes.fr}
\author[A. Yasuhara]{Akira Yasuhara} 
\address{Faculty of Commerce, Waseda University, 1-6-1 Nishi-Waseda,
  Shinjuku-ku, Tokyo 169-8050, Japan}
	 \email{yasuhara@waseda.jp}

\begin{abstract} 
  We develop a general diagrammatic theory of welded graphs, and
 provide an extension of Satoh's Tube map from welded graphs to ribbon surface-links.
As a topological application, we obtain a complete link-homotopy
classification of so-called \emph{knotted punctured spheres} in
$4$--space, by means of the $4$--dimensional Milnor invariants
introduced previously by the authors.
On the algebraic side, we show that the theory of welded graphs can be reinterpreted as a theory of Wirtinger group presentations, up
to a natural set of transformations; a group admitting such a
  presentation arises as the
fundamental group of the exterior of the surface-link obtained from the
given welded graph by the extended Tube map.
Finally, we address the injectivity question for the Tube map,
  identifying a new family of local moves on welded links, called
  $\Upsilon$ moves, under which the (non extended) Tube map is invariant. 
\end{abstract} 



\maketitle

\section*{Introduction}

A \emph{surface-link}\index{surface-link} in $4$--space is the image of an embedding of a compact oriented surface in $\R^4$ up to ambient isotopy. 
The first examples of non trivial surface-links were spun links, described in the mid-twenties by Artin \cite{cEmilletueur};  
these are obtained as the trace of the $2\pi$--rotation of a tangle in a half $3$--dimensional (3D) subspace of $\R^4$ around the
boundary of this subspace. 
Such links happen to be \emph{ribbon} surface-links, meaning that they bound immersed 3D handlebodies in $\R^4$ with only
\emph{ribbon singularities}, locally shaped as the 3D thickening of an arc passing transversally through a $3$--ball. 
Not all surface-links are ribbon, but they form an important subclass which has been
extensively studied in the late sixties \cite{Yaji,yanagawa1,yanagawa2,yanagawa3}, when the systematic study of surface-links began. 
In several ways, ribbon surface-links are easier to handle than general surface-links.  
Indeed, they admit \emph{simple projections}, which are projections in
$\R^3$ with only double points --- while a generic projection of a
surface-link contains in general triple and branch points; for embedded spheres, being ribbon is actually equivalent to having a simple projection \cite{Yaji}. 
Ribbon surface-links also admit 3D \emph{ribbon presentations}, see e.g. \cite{Maru}, which are certain ribbon-immersed disks in $3$--space. 
But despite their simplicity, ribbon surface-links can have a central role in the study of surface-links, such as in \cite{AMW} where $2$--dimensional string links  are classified up to link-homotopy using the fact that any such surface-link is link-homotopic to a ribbon one (a result which is generalized by Proposition \ref{prop:Ribbon}, see below). 

In 2000, ribbon surface-links received renewed interest when Satoh
extended Yajima's Tube map \cite{yajima} ---  
a construction which ``inflates''  classical knot diagrams into
ribbon surface-link diagrams --- to \emph{welded knotted objects}, and showed that this yields a surjective map from welded knots and knotoids to ribbon tori and spheres \cite{Satoh}; he also showed that the Tube map coincides with Artin's spun map on classical knots.  
Welded knotted objects were first introduced by Fenn--Rimanyi--Rourke \cite{FRR} for the $4$--dimensional study of
motions of unknotted circles in $\R^3$, and can be seen as a quotient of the virtual (or, equivalently, Gauss) diagram theory. 
Via Satoh's Tube map, welded knotted objects provide a $2$--dimensional description of ribbon surface-links, 
which has been used effectively in several recents works \cite{ABMW,AMW,arrow,MY7}. 

However, welded links and string links cannot describe all type of ribbon
surface-links, as it restricts to surfaces with connected components of genus 0 or 1, with
0 or 2 boundary components. To overcome this issue, the present chapter develops
a theory of \emph{welded graphs}  which generalize  welded knots, links and string links. We observe that the Tube map extends naturally to welded graphs (Proposition \ref{prop:Tube}), thus allowing for the study of the whole class of ribbon surface-links, of any topological type. 
As a concrete topological application, we obtain the following
classification result. 
\begin{theo*}[Theorem~\ref{thm:classif}] 
  Knotted punctured spheres are classified, up to link-homotopy, by non repeated  $4$--dimensional Milnor invariants.
\end{theo*}
\noindent 
In this statement,  knotted punctured spheres are properly and smoothly embedded unions of punctured spheres whose boundary
is a fixed unlink (or slice link, see Remark \ref{rem:ark}), forming a vast class of surface-links which contains Le Dimet's linked disks \cite{LeDimet} and $2$--string links studied in \cite{ABMW,AMW}. 
Link-homotopy is Milnor's classical notion of continuous deformation leaving distinct components disjoint at all
time \cite{Milnor},  and  the classification is given using the $4$--dimensional Milnor invariants of surface-links developed in \cite{cutAMY}, where this result was first announced; see Section \ref{sec:Topology} for details. 
\saut

In a different direction, we use welded graphs to address the
injectivity problem for the Tube map on welded knotted objects. 
An important open problem is indeed to determine the kernel of this map. 
The Tube map is known to be injective for welded braids \cite{BH} and
to have a non-trivial kernel in the link case \cite{Blake,IK},
but the general case remains widely open. 
As developed in this chapter, one advantage of the notion of welded graphs is that 
it naturally encompasses all the known moves on welded objects for which the Tube map is invariant. 
Actually, using this setting, we identify a large family of local
moves on welded (string) links, the \emph{$\Upsilon$ moves}  
(Definition \ref{def:HulaHoop}), under which the Tube map is invariant: 
\begin{prop*}[Proposition \ref{prop:hh}]
  The Tube map is invariant under $\Upsilon$ moves. 
\end{prop*}
\noindent These $\Upsilon$ moves do not seem to be induced by usual
welded moves, which leads us to conjecture that they provide 
general obstructions for the injectivity of the Tube map on welded (string) links. 
\saut 

Welded graphs already appeared in the literature \cite{KM,BND}, in closely related forms. 
Our approach is in comparison deliberately more combinatorial, 
given in terms of oriented graphs with edges labeled
by elements in the free group generated by the vertices. 
This highlights a striking relationship between welded objects and \emph{Wirtinger presentations}, which are finite presentations where each relation identifies a generator with a conjugate of another one. 
A celebrated result of Wirtinger provides such a presentation for the
fundamental group of the complement of a (string)link described by a diagram, 
and it is well-known that this procedure extends to welded knots and (string) links.
One of the main claims of the present chapter is that the theory of welded
graphs can actually be recasted as a theory of Wirtinger presentations
up to some natural transformations which preserve the associated
group; see Theorem \ref{prop:SameSame} for a precise statement. 
Moreover, this correspondence is naturally compatible with the extended Tube
map, providing a nice connection between combinatorial,
algebraic and topological objects, illustrated by the following result:
\begin{prop*}[Proposition \ref{prop:TubePi1}]
   For any welded graph $\Gamma$, the associated Wirtinger group $G(\Gamma)$ is isomorphic to the fundamental group $\pi_1\big(X^4\setminus \Tube(\Gamma)\big)$ of the exterior of the surface-link $\Tube(\Gamma)$.
 \end{prop*}
 
More generally, these Wirtinger presentations associated with welded graphs, yield  a whole \emph{peripheral system}, which also corresponds to a topological peripheral system for the associated ribbon surface-links. 
As a matter of fact, the above-mentioned $4$--dimensional Milnor
invariants are extracted from this peripheral system, and the above-stated
link-homotopy classification of knotted punctured spheres actually
follows from a more general diagrammatic result. 
Specifically, we introduce a notion of \emph{self-virtualization} equivalence (sv) for welded graphs, which is the welded counterpart of link-homotopy for ribbon
surface-links, and we prove the following, where \emph{reduced} refers to a
quotient of the peripheral system (see Section \ref{sec:reduit}):
\begin{theo*}[Theorem \ref{prop:RedClassification}]
  Two welded graphs are sv--equivalent if and only if they have
  equivalent reduced peripheral systems.
\end{theo*}

The chapter is organized as follows. 
In Section \ref{sec:WeldedGraphs}, welded graphs are introduced, and the relationship with welded links and string links is clarified;  the connection with ribbon surface-links via an extension of Satoh's Tube map is also presented. 
In Section \ref{sec:Wirtinger}, we introduce Wirtinger presentations and their associated peripheral systems, and the relation to welded graphs is discussed. 
We also introduce the notion of self-virtualization, and prove our classification result for welded graphs up to self-virtualization.  
In Section \ref{sec:Topology} we return to topology;  we prove that every knotted punctured spheres is link-homotopic to a ribbon one, which leads to the classification of these objects up to link-homotopy.
Finally, in Section \ref{sec:delamort}, we give the proof of Theorem \ref{th:Embed} ; further technical results are also shown in Section \ref{sec:ExtraMoves}.

\subsection*{Notation}
For every element $a$ of a group $G$, the inverse of $a$ can be
denoted by $\overline{a}$ or $a^{-1}$; and for every $a,b\in G$, we set $a^b=\overline b a
b$ and $[a,b]=\overline bab\overline a$. We denote by $F( a_1,\ldots,a_r)$ the free group on
the generators $a_1,\ldots, a_r$.

\begin{acknowledgments}
The first, resp. second, author is partially supported by the project SyTriQ (ANR-20-CE40-0004), resp. the project AlMaRe (ANR-19-CE40-0001-01), of the ANR.
The third author is supported by the JSPS KAKENHI grant 21K03237.
\end{acknowledgments}

\section{Welded graphs}\label{sec:WeldedGraphs}

\subsection{Welded knotted objects}\label{sec:on}
Let us start with a quick review of the ``usual'' welded theory. This theory was initially given in
terms of virtual diagrams \cite{Kauffman}, but we will consider here the Gauss diagram
point of view; see for instance \cite{ABMW} for a correspondance
between the two approaches.

\emph{Gauss diagrams} were first introduced as a pictural form for the Gauss code representation of
knots. They  are abstract and signed arrows whose endpoints
are embedded in a compact oriented $1$--dimensional manifold  
that shall be denoted by $M$ hereinafter; see Figure
\ref{fig:WeldedLinks} for examples.
\begin{figure}
  \[
    \begin{array}{ccc}
      \dessin{2.81cm}{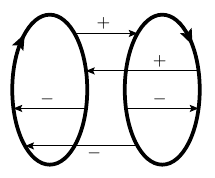}&\hspace{1cm}&\dessin{2.81cm}{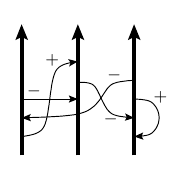}\\
      \textrm{link case}&&\textrm{string link case}
    \end{array}
    \]
  \caption{Some Gauss diagrams}
  \label{fig:WeldedLinks}
\end{figure}

\emph{Welded knotted objects}\index{welded knotted objects} are defined as equivalence classes of Gauss
diagrams up to the following \emph{welded moves} (here, and in subsequent figures, $\varepsilon$ and $\eta$ are two arbitrary signs):
\[
    \begin{array}{ccc}
      \dessin{1.875cm}{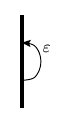} \stackrel{\textrm{Reid1}}{\longleftrightarrow}  \dessin{1.875cm}{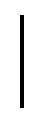} 
&\hspace{1cm}&
\dessin{1.875cm}{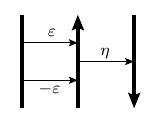} \stackrel{\textrm{Reid3}}{\longleftrightarrow}  \dessin{1.875cm}{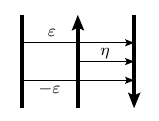}\\
\dessin{1.875cm}{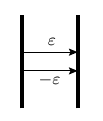} \stackrel{\textrm{Reid2}}{\longleftrightarrow}  \dessin{1.875cm}{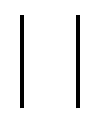} 
      & \hspace{1cm}&
\dessin{1.875cm}{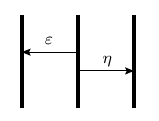} \stackrel{\textrm{TC}}{\longleftrightarrow}  \dessin{1.875cm}{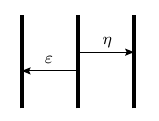}
    \end{array}.
  \]
  We stress that a TC move allows one to freely
  exchange the relative position of two adjacent arrow
  tails on $M$; TC actually stands for \emph{Tails Commute}.
  Figure \ref{fig:exchange} represents the \emph{exchange moves}, which are consequences of Reid2 and Reid3 moves (see \cite[Fig.~4.1]{ABMW}), 
  and which similarly  allow to exchange the relative positions of two
  adjacent arrow endpoints, up to the addition of a pair of
  oppositely-signed arrows.

    \begin{figure}[b]
\[
  \dessin{1.875cm}{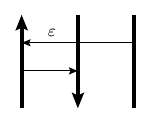} \leftrightarrow \dessin{1.875cm}{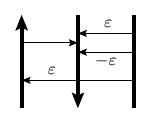}
  \hspace{2cm}
    \dessin{1.875cm}{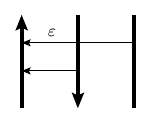} \leftrightarrow \dessin{1.875cm}{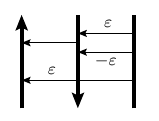}.
\]
  \caption{Exchanging arrow endpoints}
  \label{fig:exchange}
\end{figure}
  
Depending on the topological nature of the $1$--dimensional manifold $M$, one can define various types of welded knotted objects: 
\begin{itemize}
\item \emph{welded knots} correspond to the case where $M$ is a single
  copy of $S^1$, and \emph{welded links} to the case where $M$ is a
  finite union of (ordered) disjoint copies of $S^1$;
\item \emph{welded string links} correspond to the case where $M$ is a finite union of (ordered) disjoint intervals;
\item \emph{welded knotoids} are  the quotient of $1$--component welded string links, by the move which deletes or creates an arrow whose tail is adjacent to an endpoint of the interval. 
\end{itemize}

We next give some notation and observations that will be useful in
Section \ref{sec:WeldedLinks}, and that shall also smoothly drag
the above definitions towards the notion of welded graphs.

Observe that, up to ambient isotopy and TC moves, the relevant information for an arrow tail is not its precise position, but rather 
the sub-arc of $M$, bounded by two arrow heads, where this tail is located. 
This means that the information carried by an arrow can be summarized by the position of its head, together with a label indicating the sub-arc of $M$ where the tail is located; the arrow sign can in turn be given by an overlining, or a ``$-1$" exponent, on this label if the  sign is negative. 
More generally, a bunch of arrows with adjacent heads can be described
by a word obtained by concatenating the corresponding labels. 
Moreover, observe that, thanks to the TC move, bunches of adjacent arrow tails can be thought of as gathered in a single point. Such a point corresponds to  the sub-arc of $M$ containing only these tails and comprised between two arrow heads.  This allows us to think of
Gauss diagrams as a collection of such ``tails localization'' points,
connected by oriented arcs decorated with words describing the
sequence of arrow heads met on this portion of the Gauss diagram.\\
These rough observations shall be formalized in the next subsection, through the notion of welded graph.
\saut

 Finally, given an arrow $A$ of a Gauss diagram $G$, and assigning a letter to every sub-arc of $M$ comprised between two heads and containing no head in its interior, we define the \emph{head (resp. tail) conjugation of $A$ by a word $w$}, as the
insertion of the bunch of arrow heads associated with $w$ on one side
of the head (resp. tail) of $A$, and of the bunch of arrow heads
associated with $\overline w$ on the other side.
Notice that inserting such extra arrows splits some of the sub-arcs of $M$, hence changes the alphabet. 
For example, Figure \ref{fig:exchange} can be summarized by saying that an endpoint of some arrow $A_1$ can be
pushed through the adjacent head of an arrow $A_2$, at the cost of
conjugating the other endpoint of $A_1$ by the one-letter word corresponding to $A_2$.

\subsection{Definition of welded graphs}

\begin{defi}
A \emph{$w$--graph} is an oriented graph $\Gamma$ with edges decorated
by elements of the free group generated by the vertices of $\Gamma$,
and with a subset of marked vertices.\footnote{These marked vertices
    should be thought of as ``fixed'' endpoints as in string links,
  and unmarked univalent vertices as ``free'' endpoints as in knotoids.} In other words, a $w$--graph is the data of
\begin{itemize}
 \item a set $V_\Gamma$ of \emph{vertices},
 \item a set $E_\Gamma$ of \emph{edges}, made of triplets $(a,b,w)$ where $a,b\in V_\Gamma$, and $w\in
F(V_\Gamma)$ is the edge decoration,
 \item a (possibly empty) subset $V^\circ_\Gamma\subset V_\Gamma$ of \emph{marked vertices}.
\end{itemize}
The connected components of the graph are ordered, and so are 
the marked vertices on each connected component. 
An edge decorated by the unit element $1$ of $F(V_\Gamma)$ shall be called an \emph{empty edge}. 

In figures, we always represent marked vertices by white dots $\circ$, while unmarked vertices are represented by black dots $\bullet$. 

We say that a $w$--graph is of \emph{type}
$\big((m_1,b_1),\ldots,(m_\ell,b_\ell)\big)$ if the underlying graph $\Gamma$ has $\ell$ (ordered) connected components, and the $i$--th
component has exactly $m_i$ marked vertices and first Betti number
$b_i$; we simply say that it has type $(m,b)$ if $m_1=m_2=\cdots=m_\ell=m$ and $b_1=b_2=\cdots=b_\ell=b$. 
\end{defi}

\begin{figure}
  \[
    \begin{array}{ccc}
     \dessin{1.24cm}{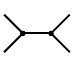}\xrightarrow{\textrm{C}} \dessin{1.24cm}{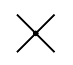} &\hspace{1cm}& \dessin{1.24cm}{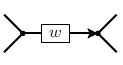}\xleftrightarrow{\textrm{OR}} \dessin{1.24cm}{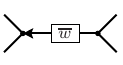}\\
      \textrm{contraction}&&\textrm{orientation reversal}
    \end{array}
  \]
  \vspace{.5cm}
  \[
    \begin{array}{c}
      \dessin{2.06cm}{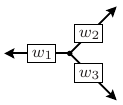}\xleftrightarrow{\textrm{S}}
      \dessin{2.06cm}{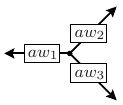}\\
      \textrm{stabilization}
    \end{array}
  \]
    \vspace{.4cm}
  \[
    \begin{array}{ccc}
      \dessin{1.24cm}{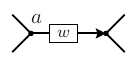}\xleftrightarrow{\textrm{R1}} \dessin{1.24cm}{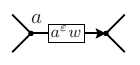}&\hspace{1cm}&\dessin{2.29cm}{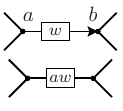}\xleftrightarrow{\textrm{R3}} \dessin{2.29cm}{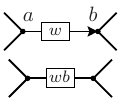}\\
      \textrm{Reidemeister 1}&&\textrm{Reidemeister 3}
    \end{array}
    \]
  \caption{Moves on w-graphs
    \\{\footnotesize Here, and in the next figure, all depicted trivalent vertices represent vertices of any valency} }
  \label{fig:WeldedMoves}
\end{figure}

We define the following moves on $w$--graphs, illustrated in Figure
\ref{fig:WeldedMoves}:
\begin{itemize}
\item \emph{contraction} (C), which removes an empty edge $(a,b,1)$ from $E_\Gamma$ and replaces
  $V_\Gamma$ by $\fract{V_\Gamma}{a\sim b}$, where $a$ and $b$ are identified, under the assumption that 
  $b\notin V^\circ_\Gamma$; 
\item \emph{orientation reversal} (OR), which replaces an edge
  $(a,b,w)$ by $(b,a,\overline w)$;
\item \emph{stabilization} (S), which adds a common prefix $a$ to all the edges connected to a given
  vertex $b$, under the assumption that all the edges are oriented
  outwards $b$, that $b\notin V^\circ_\Gamma$ and that all edge
    decorations of $\Gamma$ are in $F\big(V\setminus\{b\}\big)$;
\item \emph{Reidemeister I} (R1), which replaces an edge $(a,b,w)$ by
  $(a,b,a^{\e}w)$ for some $\e\in\{\pm1\}$;
  \item \emph{Reidemeister III} (R3),  which replaces by $wb$ the decoration $aw$ of an edge $e\in
  E_\Gamma$, under the assumption that $(a,b,w)\in
  E_\Gamma\setminus\{e\}$.
\end{itemize}

  \begin{remarque}
    The contraction move induces its inverse move, the
    \emph{expansion} move (E), which duplicates a
    vertex $v$ into $v_1$ and $v_2$, and creates a new
   empty edge between them. All the edges which
   were adjacent to $v$, have to be distributed between $v_1$ and $v_2$,
   and any occurence of $v$ in edge decorations can be freely replaced by $v_1$ or $v_2$.
  \end{remarque}

Note that the number of marked vertices,
as well as the first Betti number, are preserved by all the above
moves.
\begin{defi}
  A \emph{welded graph}\index{welded graph} is an element of the set
  \[
\wGraph:=\fract{\big\{\textrm{$w$--graphs}\big\}}{\textrm{C,OR,S,R1,R3}},
\]
and its \emph{type} is defined as the type of any of its representatives.
\end{defi}

\begin{figure}
  \[
    \begin{array}{ccc}
     \dessin{2.06cm}{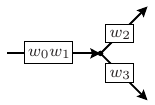}\xleftrightarrow{\textrm{P}}
    \dessin{2.06cm}{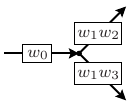} &\hspace{1cm}&\dessin{1.24cm}{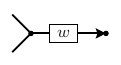}\xleftrightarrow{\textrm{\ P\ }}
                                           \dessin{1.24cm}{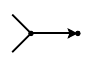}
                                                                   \\
      \textrm{push move}&&\textrm{push move (univalent case)}
    \end{array}
  \]
     \[
    \begin{array}{ccc}
     \dessin{1.24cm}{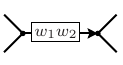}\xrightarrow{\textrm{Split}}\dessin{1.24cm}{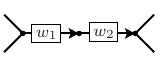}&\hspace{1cm}&\dessin{1.24cm}{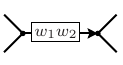}\xleftrightarrow{\textrm{R3}}\dessin{1.24cm}{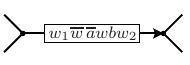}\\
                                                                                \textrm{split move}   &&\textrm{rephrased R3}
    \end{array}
  \]
    \[
    \begin{array}{c}
     \dessin{2.94cm}{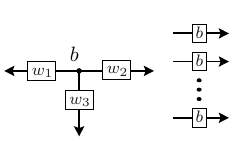}\xleftrightarrow{\textrm{S}}\dessin{2.94cm}{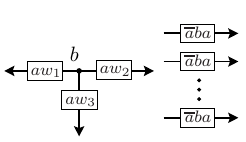}\\
      \textrm{generalized stabilization}
    \end{array}
  \]
  \caption{Extra moves on w-graphs}
  \label{fig:ExtraWeldedMoves}
\end{figure}

More admissible moves can be obtained by combining moves C, OR, S, R1 and R3.
We gather below some of these extra moves, illustrated in Figure \ref{fig:ExtraWeldedMoves}: 
\begin{itemize}
\item \emph{push move} (P), which pushes, according to the edges orientations, a word through an unmarked
  vertex that does not appear in any edge decoration; note that, in the case where the involved vertex is univalent, 
  the push move simply erases the word on the incident edge;
\item \emph{split move} (Split), which splits an edge in two and
  separates the decoration between them. Note that it induces its
    inverse move, which deletes a bivalent vertex that does not
    appear in  any edge decoration, merges the two adjacent edges and concatenates their edge 
  decorations.
\end{itemize}
Also:
\begin{itemize}
\item the Reidemeister III move can be rephrased as follows: given an edge $(a,b,w)$, 
  insert the word $\overline w\,\overline awb$ (or any cyclic 
  permutation of it, or of its inverse) within the decoration of any other edge;
\item in a stabilization move, the condition that the central vertex $b$ does not
  appear in any edge decoration can also be relaxed, at the cost of
  conjugating by the prefix $a$ every occurence of $b$ in the edge
  decorations.
\end{itemize}
We postpone to Section \ref{sec:ExtraMoves} the proof that each of these extra  moves are indeed consequences of moves C, OR, S and R3.

\subsection{Relation to welded links and string links}\label{sec:WeldedLinks}

As mentioned earlier, welded graphs are closely related to welded
links and string links. This is made explicit in Theorem \ref{th:Embed} below, which is the main result of this subsection.

\subsubsection{Statement}
We define two maps
\begin{gather*}
 \psi_L:\big\{\textrm{welded links}\big\}\to \big\{\textrm{welded
   graphs of type $(0,1)$}\big\}\\
  \psi_{SL}: \big\{\textrm{welded string links}\big\}\to \big\{\textrm{welded graphs of type $(2,0)$}\big\}
\end{gather*}
as follows. 
With a Gauss diagram $D$ for a welded link, we associate a $w$--graph $\Gamma_D$. 
The vertices of $\Gamma_D$ are all unmarked,  and are in bijection with the arrow tails of $D$. The
edges correspond to the oriented pieces of circle  bounded by two
``adjacent'' arrow tails, which are such that they contain no arrow tail in their interior. 
When running according to its orientation, each such piece of circle contains an ordered (possibly empty) sequence $h_1\cdots h_k$ of arrow heads; the decoration of the corresponding edge of $\Gamma_D$ is obtained by replacing each $h_i$ by $t_i^{\e_i}$,
where $t_i$ and $\e_i$ are respectively the tail and the sign of the
arrow which contains $h_i$. Finally, in order to reduce the number of
vertices, empty edges can be contracted. 
See the left-hand side of Figure \ref{fig:GaussToWelded} for an example. 

For a Gauss diagram $D$ of welded string link, we define $\Gamma_D$
in a similar way; each component of the $w$--graph $\Gamma_D$ now contains two marked vertices, which are in bijection with the endpoints of the intervals, and the edges correspond to the oriented pieces of interval between every pair of adjacent arrow tails and/or interval endpoints. 
See the right-hand side of Figure \ref{fig:GaussToWelded}.
\begin{figure}
  \[
    \dessin{2.81cm}{GaussLink}\leadsto \dessin{3.44cm}{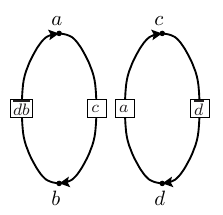}
    \hspace{1cm}
    \dessin{2.81cm}{GaussStringLink2}\leadsto \dessin{3.125cm}{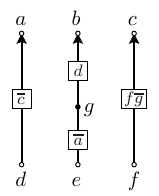}
    \]
  \caption{From Gauss diagrams to $w$--graphs}
  \label{fig:GaussToWelded}
\end{figure}

It is a straightforward exercise to check that this assignment $D\mapsto \Gamma_D$ is well defined on welded links and string links, hence 
induces the requested maps $\psi_L$ and $\psi_{SL}$. 

Moreover, as Figure \ref{fig:GaussToWelded} illustrates, the $w$--graphs obtained from the above procedure, have a very specific shape. 
This motivates the following definitions. 
\begin{defi}\label{def:cycliclinear}
A $w$--graph of type $(0,1)$ is called \emph{cyclic} if each component is homeomorphic to a circle, decomposed into bivalent unmarked vertices and oriented edges, such that each vertex has one inward and one outward oriented incident edge; see the left-hand side of Figure \ref{fig:GaussToWelded}.\footnote{ We note that cyclic $w$--graphs are reminiscent of Harary and Kauffman's arc-graphs \cite{HaraKoko}.}

A $w$--graph of type $(2,0)$ is called \emph{linear} if each component
is homeomorphic to an interval, decomposed into two univalent marked
vertices, bivalent unmarked vertices and oriented edges, such that each unmarked vertex has one inward and one outward oriented incident edge; see the right-hand side of Figure \ref{fig:GaussToWelded}. 
\end{defi}

\begin{lemme}\label{lem:PC}
The maps $\psi_L$ and $\psi_{SL}$ are surjective. 
\end{lemme}
\begin{proof}
Since the link case is strictly similar, we only give the argument in the string link case. 
We first observe that, up to push and contraction moves, any
$w$--graph of type $(2,0)$ is equivalent to a linear $w$--graph.
Indeed, each component of a $w$--graph of type $(2,0)$ is a tree with two
marked vertices; we consider the shortest path of edges between these
two marked vertices.
Next we pick some edge $e$ which is not in this path, but which
  has an endpoint $a$ on it (and a tree $T$ attached to its other
 endpoint). If $a$ is not used in any edge decoration, we can simply apply a push move 
 to turn $e$ into an empty edge, and then a contraction move to delete
 it; otherwise, we need to first push $a$ inside the
 shortest path using an expansion move, which creates an empty
edge, so that we can remove $e$ as in the former case:\footnote{Another issue would be that the vertex we are willing
  to push through is a marked vertex, but this can be safely avoided
  by pulling first the initial and final vertices away by using expansion moves.} 
\[
  \dessin{2.18cm}{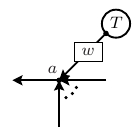}\ \xrightarrow{\ \textrm{E}\ }
  \dessin{2.18cm}{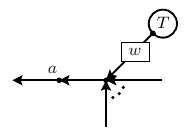}\ \xrightarrow{\ \textrm{P}\ }
  \dessin{2.18cm}{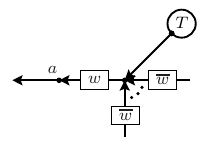}\ \xrightarrow{\ \textrm{C}\ }
  \dessin{2.18cm}{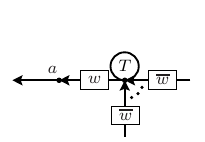}.
  \]
Applying this procedure recursively, we can delete all edges which are
not in the shortest path. Then it only remains to apply OR moves to
orient all edges coherently. 
Now, it remains to note that a linear $w$--graph $\Gamma$ admits a preimage for the map $\psi_{SL}$. 
The process is summarized in the figure below:  
\[ \dessin{.625cm}{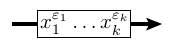} \xrightarrow{\textrm{Split}} \dessin{.625cm}{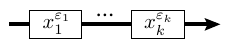}\qquad \leadsto \qquad \dessin{1.41cm}{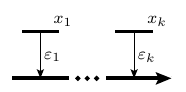}. \]
First, using Split moves, we split each word label on a linear $w$--graph
into letter labels so that each edge is decorated by the, possibly overlined, name of a vertex. 
Next, we replace  each  (non empty) edge decoration by the head of an
arrow, whose tail is attached in a neighborhood of the vertex corresponding to the letter decoration; the sign of this arrow being $-$ if the letter is overlined, and $+$ otherwise.
Note that the relative position of the tails corresponding to a same letter label is irrelevant thanks to the TC move.  
\end{proof}

\begin{remarque}
 As roughly outlined at the end of Section \ref{sec:on}, and
 formalized by the map $\psi_{SL}$, word labelings on the edges of a
 welded graph can locally be thought of as some Gauss diagram arrows. 
 In subsequent figures, we shall  use implicitly this correspondence.
Concretely, this means that we shall sometimes represent bunches of adjacent arrow heads in Gauss diagrams by some
  edge decoration words, whose signed letters correspond to the arrows signs
  and tails.  
A typical example follows in the next definition.
\end{remarque}

We now show that these maps are not injective. 
More precisely, we prove that the injectivity defect is controlled by
the following moves. 

\begin{defi}\label{def:HulaHoop}
  \begin{itemize}
  \item[]
  \item   On welded links and string links, we define the  $\Upsilon$ move as the following local move:
\[
\dessin{3.75cm}{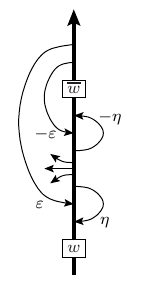}
\ \xleftrightarrow{\Upsilon}\ 
\dessin{3.75cm}{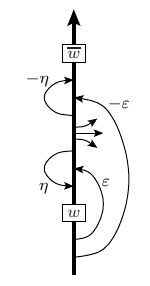},
\]
 where the three open arrows represent any number of tails whose associated heads are located outside of the picture.\footnote{Since
  the two diagrams are identical outside the represented portion, this
  means that the two bunches are ``identical'' for both diagrams; it
also means that they are not connected to any of the arrows associated with $w$.}
\item On welded links, we define the \emph{global reversal} (GR) move, as
  the move which reverses the orientation of a component, and reverses the
  sign of all arrows whose head is attached to this component.
   \end{itemize}
\end{defi}
\begin{remarque}\label{rem:versions}
We stress that the  $\Upsilon$ move actually is a \emph{family} of moves on welded (string) links. 
Indeed, there are multiple versions of the  $\Upsilon$ move, depending on the word $w$, the signs $\varepsilon$ and $\eta$, and the number of arrow tails (schematically represented by the three open arrows). See for example Section \ref{sec:ExtraMoves2}, where we consider the case  $w=1$, $\eta=-\e$ and involving only the four self-arrows shown in the above figure (and no arrow tail). 
\end{remarque}
The following is the main result of this section. 
As we shall see in Subsection \ref{sec:ribbon}, it is also relevant from the topological point of view.  
\begin{theo}\label{th:Embed}
  \begin{itemize}
  \item[]
  \item The map $\psi_L$ induces a bijection between welded graphs of type $(0,1)$ and welded
    links up to $\Upsilon$ and global reversal moves.
      \item The map $\psi_{SL}$ induces a bijection between welded graphs of type $(2,0)$ and welded
    string links up to $\Upsilon$ moves.
  \end{itemize}
\noindent In particular, the set of welded links and string links, up to $\Upsilon$ and global reversal
  moves, embed into the set of welded graphs.
\end{theo}

For ease of reading, the proof of Theorem \ref{th:Embed} is
  postponed to Section \ref{sec:delamort}.

\subsection{Relation to ribbon surface-links}\label{sec:ribbon}

One important fact about welded links and string links, is that they are closely related to
ribbon surface-links whose components are, respectively, tori and
annuli,  
in the sense that there exists a surjective map from the
former to the latter.
Welded graphs are similarly related to all kind of ribbon
surface-links, 
thus providing a more flexible combinatorial theory for the study of general ribbon surface-links.
In what follows, $X^4$ denotes a  given ambient $4$--manifold in $\big\{S^4,\R^4,B^4\big\}$. 
\begin{defi}
  \begin{itemize}
  \item[] 
  \item A \emph{ribbon handlebody} is an immersion in $X^4$ 
  of a union $H$ of oriented
  solid handlebodies up to ambient isotopy,  such that
  \begin{itemize}
  \item the singular set is a finite union of \emph{ribbon disks}; 
    here a ribbon disk is a disk $D$ of transverse double points with one preimage $D_{\textrm{con}}$,
    called \emph{contractible}, embedded in the interior of $H$, and
    one preimage $D_{\textrm{ess}}$, called \emph{essential},
    properly  embedded in $H$;
  \item  if $X^4=B^4$,
    $H\cap\partial B^4$ is a disjoint union of embedded disks.
  \end{itemize}
  \item A \emph{ribbon surface-link}\index{ribbon surface-link} is an oriented surface-link
    which bounds\footnote{In case $X^4=B^4$, the boundary has to be understood as the
      closure of 
      $\partial H\setminus (H\cap\partial B^4)$.} a ribbon handlebody; the latter is called a \emph{ribbon filling} of the ribbon surface-link.
  \end{itemize}
\end{defi}
We note that a given ribbon surface-link, may admit several different ribbon fillings. 

By definition, the essential preimages of the ribbon disks of a ribbon
handlebody $H$, cut $H$ into 3--dimensional pieces, that we call \emph{chambers}. Each chamber has the
topology of a 3--ball with a finite number of handles attached, 
and a finite number of disks removed from its boundary. These removed 
disks correspond to essential preimages, and each of them is shared by 
two chambers (possibly twice the same). Moreover, each 
contractible preimage of a ribbon disk of $H$ embeds into the interior 
of one of these 
chambers. \\

We say that a $w$--graph is \emph{reduced} if all empty edges are loops, all non empty edges are decorated by
one-letter words, and all marked vertices are univalent. We can define a map
\[
C:\big\{\textrm{ribbon handlebodies}\big\}\to
\fract{\big\{\textrm{reduced $w$--graphs}\big\}}{OR}
\]
as follows. If $H$ is a ribbon handlebody, then $C(H)$ has:
\begin{itemize}
\item one unmarked vertex for each chamber and, if $X^4=B^4$, one marked vertex for
  each connected component of $H\cap \p B^4$;
\item as many empty loops attached to each vertex as the genus
  of the corresponding chamber;
\item one arbitrarily-oriented edge $e_D$ for each ribbon disk $D$: the edge
  connects the two vertices corresponding to the two chambers that are adjacent to $D_{\textrm{ess}}$, 
  and it is decorated by the letter corresponding to the (vertex associated with the) chamber containing $D_{\textrm{con}}$;
  this letter is overlined if and only if, for any given $x_0\in D$, the
  concatenation of (the push-forward of) a positive basis of
  $T_{x_0^{\textrm{con}}}H$, where $x_0^{\textrm{con}}$ is the
  preimage of $x_0$ in $D_{\textrm{con}}$, with (the push-forward of)
  a normal vector for $D_{\textrm{ess}}$ at $x_0^\textrm{ess}$, where $x_0^{\textrm{ess}}$ is the
  preimage of $x_0$ in $D_{\textrm{ess}}$, pointing into the chamber
  corresponding to the target vertex of $e_D$, is a negative basis of $X^4$.
\end{itemize}
\begin{lemme}\label{lem:CBij}
The map $C$ is a bijection. 
\end{lemme}
\begin{proof}
This follows from the exact same argument as in the proof of \cite[Prop. 3.7]{Aud_HdR}, which deals with the string link case. 
Roughly speaking, the proof is based on the fact that, given embedded
$4$--balls in $S^4$ with a finite number of paired embedded
$3$--balls on their boundary, there is a
unique way, up to ambient isotopy, to connect these $3$--balls with
disjoint $1$--handles. These $1$--handles are indeed characterized by
their $1$--dimensional framed core; and since the framing takes values
in $S^2$, the unit sphere of the normal bundle of the core, which is
simply connected, all framing are isotopic. 
\end{proof}

Next, consider the map 
\[
\partial:\big\{\textrm{ribbon handlebodies}\big\}\to
\big\{\textrm{ribbon surface-links}\big\}, 
\]
which sends a ribbon handlebody to its boundary. 
\begin{prop}\label{prop:Tube}
The composite $\partial\circ C^{-1}$, induces a well-defined, surjective map 
\[
\Tube:\big\{\textrm{welded graphs}\big\}\to \big\{\textrm{ribbon surface-links}\big\}, 
\]
sending a type $\big((m_1,b_1),\ldots,(m_\ell,b_\ell)\big)$ welded graph to an embedded $\sqcup_{i=1}^\ell \Sigma_i$, where $\Sigma_i$ is a genus $b_i$ surface with $m_i$ disjoint disks removed. 
\end{prop}
We shall sometimes call this map the \emph{generalized Tube
  map}\index{Tube map (generalized)},
since it naturally coincides with the known Tube map for welded (string) links, see Remark \ref{rem:genius} below.  
\begin{proof}
 Denote by $\widetilde T$ the composite $\partial\circ C^{-1}$. 
 Note that, thanks to contraction and split moves, every
$w$--graph is canonically equivalent to a reduced $w$--graph. 
Hence $\widetilde T$ naturally extends to all $w$--graphs. 
Let us check now that $\widetilde T$ is invariant under C, S, R1 and R3 moves.

  Invariance under a contraction move is clear, and since the moves are local, invariance
  under R1 and R3 moves follow from the same arguments as for the original Tube map; see for instance the explicit ribbon filling changes argument given in  the proof of \cite[Prop. 2.5]{IndianPaper}. 
  Let us now consider the stabilization move, which inserts a letter $x_0$ on the edges incident to a vertex $v$. 
  Via $\widetilde T$, the letter $x_0$ corresponds to a chamber $C_0$, and $v$ corresponds
  locally to a punctured sphere $S$ filled by a $3$--ball $C$; up to ambient isotopy, $C$ and $C_0$ can be assumed to stand close to each other.  
  Now, we can create\footnote{See, for instance the local ribbon filling change given in \cite[Fig. 16]{IndianPaper}.} two oppositely-signed ribbon disks between $C$ and $C_0$, such that the boundaries of the essential preimages are both
  parallel to a given puncture of $S$ -- see the left-hand side of the figure below, which represents the chamber $C$. 
  Then we push one of these ribbon disks through $C$, until it splits into ribbon disks parallel to every other
  punctures of $S$: 
    \[
\dessin{3.125cm}{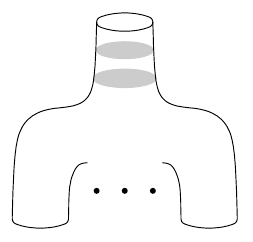}\ \xrightarrow{\ \dessin{1.25cm}{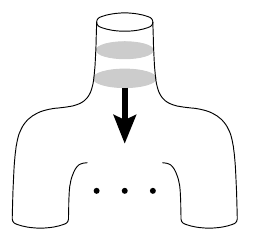}\ }\ \dessin{3.125cm}{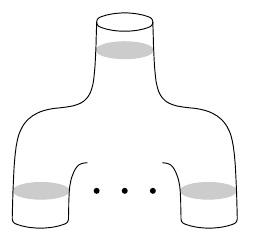}. 
    \]
  The result is the image by $\widetilde T$, of the stabilization
  move. This completes the proof that $\widetilde T$ is invariant
  under all moves on welded graphs, so that the map $\widetilde T$
  induces the desired map Tube.
  
Surjectivity of Tube is clear, since any ribbon surface-link $L$ admits at least one ribbon filling $R_L$ by definition, $C(R_L)$ provides
then a welded graph $\Gamma_L$ such that $\Tube(\Gamma_L)=L$. 
The latter part of the statement, concerning the types of welded graphs, just follows by construction.
\end{proof}

\begin{remarque}\label{rem:genius}
As the name suggests, the definition of our generalized Tube map follows closely the original construction, as given in
\cite{IndianPaper,Aud_HdR}. Specifically, by construction we have that $\Tube\circ\psi_L$ and $\Tube\circ\psi_{SL}$, coincide with the original Tube maps for welded links and string links, respectively.

As a matter of fact, the generalized Tube map can be defined directly, in terms of broken
surface diagrams, as in \cite{Satoh} (we shall not recall here the langage of broken surface diagrams, but refer instead the reader to \cite{CKS}). 
This is illustrated by the following local pictures: 
  \[
  \dessin{1cm}{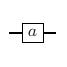}\rightarrow\dessin{2.2cm}{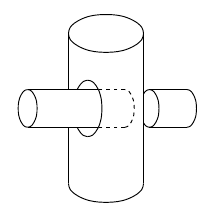}
  \hspace{1cm}
   \dessin{.8cm}{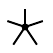}\rightarrow\dessin{2.2cm}{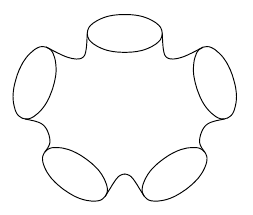}
  \hspace{1cm}
  \dessin{.8cm}{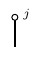}\rightarrow\dessin{1.8cm}{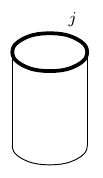}.
  \]
Here, the $j$--labeled bold circle on the rightmost picture
  denotes the $j$--th component of a prescribed unlink in
  $\partial B^4$. In general, for an arbitrary welded graph $\Gamma$, we obtain a diagram for $\Tube(\Gamma)$ by connecting the above pieces of broken surface diagrams with arbitrary embedded annuli, as prescribed by the welded graph $\Gamma$. 
\end{remarque}

The Tube map, as defined by Satoh on welded knotoids and knots \cite{Satoh}, and extended later to welded links and
string links \cite{ABMW}, served as a tool for studying ribbon $2$--knots, tori and $2$--string links. 
By Proposition \ref{prop:Tube}, welded graphs now provide a general setting for studying the whole class of ribbon surface-links.

In contrast, injectivity of the Tube map has remained an intriguing
question. In \cite{BH}, the Tube map has been proved to be injective on welded
braids, but Global Reversal moves were proved to be in its kernel for links \cite{Blake,IK}, and 
the question remains widely open for welded 
string links. In view of Theorem \ref{th:Embed}, we have the following. 
\begin{cor}\label{prop:hh}
  The Tube map is invariant under  $\Upsilon$ moves.
\end{cor}
We conjecture that, except for some very specific cases (see Section \ref{sec:ExtraMoves2}), 
the $\Upsilon$ moves are not induced by the usual welded moves, and hence provide a new large class of
non trivial moves in the kernel of the Tube map for welded links and string links.
As a matter of fact, welded graphs intrinsically contain all the
known
moves under which the Tube map is invariant.

\section{Wirtinger groups, peripheral systems and reduced quotient}\label{sec:Wirtinger}

\subsection{Wirtinger groups}
The fundamental group of the complement is an essential invariant in the topological study of embedded manifolds in codimension 2.
In the case of knots and links, any diagram provides a so-called Wirtinger presentation for this group. 
This diagrammatical point of view on the fundamental group generalizes to other situations, such as surface-links \cite{CKS} or welded knot theory \cite{Kauffman}.
The goal of the present section is to show that the theory of welded graphs can in fact be interpreted as a theory of Wirtinger presentations.
\begin{defi}
  A \emph{Wirtinger presentation} is a group presentation
  \[
\big\langle x_1,\ldots,x_k\ \big|\ R_1,\ldots,R_s\big\rangle
\]
with a finite number of generators, and a finite number of relations of the form $R=\overline x_jx_i^w$, where
$x_i$ and $x_j$ are two generators and $w$ is an element of the free group $F(x_1,\ldots,x_k)$.
A \emph{Wirtinger group}\index{Wirtinger group} is a group admitting a Wirtinger presentation. 
\end{defi}

One way to describe a Wirtinger presentation is by a decorated
oriented graph with one vertex per generator, and one $w$--decorated edge from $x_i$
to $x_j$ for each relation $\overline x_jx_i^w$. Such a decorated graph is
nothing but a $w$--graph, with no marked vertex. 
Note that performing an R1 move on this graph does not change the Wirtinger
presentation since in $F(x_1,\ldots,x_k)$, we have $\overline
x_jx_i^{x_iw}=\overline x_jx_i^{\overline x_iw}=\overline
x_jx_i^w$. More generally, we have the following.
 \begin{prop}\label{prop:SameSame}
  Welded graphs with no marked vertex are in one-to-one correspondence
  with Wirtinger presentations up to the following operations:
  \begin{itemize}
  \item $\big\langle x_1,\ldots,x_k\ \big|\
    R_1,\ldots,R_s,\overline x_{i_2}x_{i_1}\big\rangle
    \leftrightarrow \big\langle x_1,\ldots,x_{i_2-1},x_{i_2+1},\ldots,x_k\ \big|\
    {R_1}_{|x_{i_2}\to x_{i_1}},\ldots, {R_s}_{|x_{i_2}\to
      x_{i_1}}\big\rangle$;
  \item $\big\langle x_1,\ldots,x_k\ \big|\
    R_1,\ldots,R_s,\overline x_{i_1}x_{i_2}^w\big\rangle\leftrightarrow\big\langle x_1,\ldots,x_k\ \big|\
    R_1,\ldots,R_s,\overline x_{i_2}x_{i_1}^{\overline w}\big\rangle$;
  \item $\big\langle x_1,\ldots,x_k\ \big|\
    R_1,\ldots,R_s\big\rangle\leftrightarrow\big\langle x_1,\ldots,x_k\ \big|\
    {R_1}_{|x_{i_1}\to x_{i_1}^\alpha},\ldots,{R_s}_{|x_{i_1}\to
      x_{i_1}^\alpha}\big\rangle$;
  \item $\big\langle x_1,\ldots,x_k\ \big|\
    R_1,\ldots,R_s,\overline x_{i_2}x_{i_1}^w ,\overline x_{i_4}x_{i_3}^{x_{i_1}w},\big\rangle\leftrightarrow \big\langle x_1,\ldots,x_k\ \big|\
    R_1,\ldots,R_s,\overline x_{i_2}x_{i_1}^w ,\overline x_{i_4}x_{i_3}^{wx_{i_2}}\big\rangle$
  \end{itemize}
where $R_{|a\to b}$ denotes the word obtained from the word $R$ by replacing all occurrences of $a$ by $b$, and where $\alpha$ is any generator or inverse of  a generator.
\end{prop}
\begin{proof}
  It is straightforwardly checked that the first operation corresponds
  to a contraction move, the second  one to an orientation reversal, the third one 
  to a generalized stabilization, and the last one to a Reidemeister 3 move.
\end{proof}

Observe that, up to isomorphism, the four operations in Proposition \ref{prop:SameSame}, preserve the underlying group. 
For any welded graph $\Gamma$, there is
hence an associated Wirtinger group, which we denote by $G(\Gamma)$,
given by the presentation associated with $\Gamma$ (after
  having unmarked all marked vertices). 
As a matter of fact, this group is  the fundamental group of the complement of the associated ribbon surface-link via the generalized Tube map:

\begin{prop}\label{prop:TubePi1}
  For every welded graph $\Gamma$, the group $G(\Gamma)$ is isomorphic to  $\pi_1\big(X^4\setminus \Tube(\Gamma)\big)$.
\end{prop}

\begin{proof}
  Let $\Gamma$ be a $w$--graph representative of a given welded graph. 
  As noted in the proof of Proposition \ref{prop:Tube}, this representative $\Gamma$ can be assumed to be reduced, so that we can consider the associated ribbon handlebody $H=C(\Gamma)$. Recall that $H$ is cut into chambers by ribbon disks, and
  that $\Tube(\Gamma)=\p H$. We observe that the desired presentation for 
  $\pi_1\big(X^4\setminus \Tube(\Gamma)\big)$ can be given in terms of
  the ribbon disks and the chambers of $H$. This follows from the three following facts:
  \begin{itemize}
  \item $H$ retracts to a $1$--dimensional graph in $X^4$, so that
    $X^4\setminus H$ is connected and simply connected;
  \item since $H$ is $3$--dimensional, and since ribbon disks are $2$--dimensional, a generic loop in $X^4\setminus
    \Tube(\Gamma)$ avoids all ribbon disks and intersects $H$ in a finite number of transverse points. 
    Moreover, two loops intersecting the same chamber once with same sign are homotopic. Hence $\pi_1\big(X^4\setminus
    \Tube(\Gamma)\big)$ is generated by a set of loops which is in bijection with chambers, 
    a chamber $c$ corresponding to the class of a loop intersecting $H$ exactly once, in the interior of $c$;
  \item a generic homotopy of loops  in $X^4\setminus
    \Tube(\Gamma)$ intersects the ribbon disks in a finite number of
    transverse points. Each intersection corresponds to a
    Wirtinger relation between the two generators associated with the
    chambers on either side of the intersected ribbon disk, one being
    equal to the conjugate of the other by the generator associated
     with the chamber containing the contractible preimage of the
    ribbon disk.
  \end{itemize}
Since ribbon disks and chambers of $H$ are in bijection with edges and
vertices of $\Gamma$, this presentation is nothing but the Wirtinger presentation of $G(\Gamma)$.
\end{proof}

In the rest of this section, we shall see that a Wirtinger presentation carries more information than just the associated group.

\subsection{Peripheral system}

The fundamental group of a codimension 2 embedding complement can be endowed with the data of so-called peripheral elements, 
which are given by loops staying essentially close to the embedded manifold. 
In the case of links, these elements are meridians and longitudes, and they form, together with the group, the peripheral system. 
Implicitly used by Dehn in 1914 to distinguish the right and the left-handed trefoil, peripheral systems were only
formally introduced by Fox \cite{Fox} in the early fifties and proved to be a complete
invariant for links by Waldhausen \cite{Waldo} in 1968.

This notion of peripheral system has a combinatorial counterpart for welded graphs, which we introduce now. 
Let $\Gamma$ be a $w$--graph, with connected components  $\Gamma_1,\ldots,\Gamma_\ell$. 
Denote by $b_i$ the first Betti number of $\Gamma_i$ and by $m_i$ its number of marked vertices. 

\begin{nota}\label{cvraimentjusteunenota}
  A combinatorial path between two vertices $v_1$ and $v_2$ on $\Gamma$, is a finite sequence of adjacent edges starting at $v_1$
  and ending at $v_2$. Note that, on each edge, the orientation given by the   $w$--graph and the orientation given by the path may differ. 
  With any combinatorial path $\mathfrak p=(e_1,\ldots,e_k)$, one can associate the element
  $w_{\mathfrak p}\in G(\Gamma)$ defined by the concatenated word $w_1^{\e_1}\ldots w_k^{\e_k}$, where $w_i$ is the label of $e_i$
and $\e_i\in\{\pm1\}$ is 1 if and only if the orientations of $e_i$ as an edge of $\Gamma$ and as a path segment of $\mathfrak p$ coincide.
\\
A combinatorial path corresponds in particular to a continuous path between two vertices and, reciprocally, any continuous path $\gamma$ between
two vertices is homotopic relative to the endpoints, to a combinatorial path $\mathfrak p_\gamma$. 
Note that two path representatives for $\gamma$ differ only by some back and forth along some edges, so that the associated element in $G(\Gamma)$ is well-defined.
We shall hence use simply the notation $w_\gamma$ for the element $w_{\mathfrak p_\gamma}\in G(\Gamma)$.
\end{nota}

\begin{remarque}\label{rk:Conjugate}
  It follows from iterated uses of Wirtinger relations that, if $\gamma$ is
  a path between two vertices $v_1,v_2\in\Gamma$, then
  $w_\gamma v_2=v_1w_\gamma$ in $G(\Gamma)$. This implies, in particular,
  that, for every $i\in\{1,\ldots,\ell\}$, any two vertices of
  $\Gamma_i$, seen as elements of $G(\Gamma)$, are conjugate, and that if $\gamma$ is a loop based at some vertex $v$,
  then $v$ and $w_\gamma$ commute. 
\end{remarque}

\begin{defi}
\emph{Meridians} of $\Gamma$ are defined as the generators of $G(\Gamma)$, that is, the vertices of $\Gamma$. 
Vertices on the $i$--th component of $\Gamma$ are sometimes more precisely called \emph{$i$--th meridians}.

A \emph{basing} for $\Gamma$ is the data $\big((\mu_i)_i,\{l^i_1,\cdots ,l^i_{m_i}\}_{i,j},(a^i_1,\cdots,a^i_{n_i})_{i}\big)$ where, for every $i\in\{1,\ldots,\ell\}$: 
\begin{itemize}
\item $\mu_i$ is a choice of a meridian on $\Gamma_i$;
\item $l^i_1,\ldots,l^i_{b_i}$ are $b_i$ loops on $\Gamma_i$ based at $\mu_i$, that generate $\pi_1(\Gamma_i,\mu_i)$;
\item $a^i_1,\ldots,a^i_{m_i}$ are $m_i$ paths that connect $\mu_i$ to each marked vertex of $\Gamma_i$.
\end{itemize}
\end{defi}
  Note that, since connected components and marked points of $\Gamma$
  are ordered, meridians and $a_i^j$ paths  are also ordered,
  whereas $l_i^j$ loops are not.
\begin{remarque}\label{ohlabelleremarque}
If $\Gamma$ has a marked vertex on each component, then there is a canonical choice for the chosen meridians, namely the minimal marked vertex on each component.
\end{remarque}

\begin{defi}
  Let $\Gamma$ be endowed with a basing $\big((\mu_i)_i,\{l^i_j\}_{i,j},(a^i_j)_{i,j}\big)$. 

  For every $i\in\{1,\ldots,\ell\}$, we define
  \begin{itemize}
  \item \emph{preferred $i$--th loop-longitudes}, as
    the elements $\mu_i^{-|\lambda^i_j|_i}\lambda^i_j\in G(\Gamma)$, for $j\in\{1,\ldots,b_i\}$;
  \item \emph{preferred $i$--th arc-longitudes}, as the elements
    $\mu_i^{-|\alpha^i_j|_i}\alpha^i_j\in G(\Gamma)$, for $j\in\{1,\ldots,m_i\}$; 
  \end{itemize}
  where $\lambda^i_j$, resp. $\alpha^i_j$, denote the element $w_{l^i_j}$, resp. $w_{a^i_j}$, of $G(\Gamma)$ 
  and where $|w|_i$  is the algebraic number of letters in the word $w\in G(\Gamma)$ that correspond to vertices in
    $\Gamma_i$. A \emph{longitude} is either a (preferred)  loop or an arc-longitude.
  \end{defi}
 As noted above, meridians and arc-longitudes are gathered in ordered tuples, whereas loop-longitudes are elements of unordered sets.
\begin{prop}\label{prop:WeldedBasingIso}
    Let $\Gamma'$ be a $w$--graph obtained from $\Gamma$ by a welded move. There is a canonical isomorphism
  $\varphi:G(\Gamma)\to G(\Gamma')$, which sends any  basing of  $\Gamma$ to a basing for $\Gamma'$. 
\end{prop}
\begin{proof}
  The statement is clear for contraction, orientation reversal,
  stabilization and Reidemeister 3 moves. A Reidemeister 1 move
  applied on an edge $e$  introduces some letter $v^{\pm 1}$ in the longitude word $w_\gamma$ associated
  with any path $\gamma$ passing through $e$,  and changes $|w_\gamma|_i$ by $\pm 1$. Using Remark \ref{rk:Conjugate},
  $v^{\pm 1}$ can be replaced by a $\mu_i^{\pm1}$ factor at the beginning of the longitude
  word $w_\gamma$, so that it compensates  the change of $|w_\gamma|_i$.
\end{proof}

\begin{defi}\label{def:PeripheralSystems}
A \emph{peripheral system}\index{peripheral system} for a welded graph $\Gamma$ is the data
$$\big(G(\Gamma),(\mu_i)_i,\{w_{l^i_j}\}_{i,j}, (w_{a^i_j})_{i,j}\big),$$ 
\noindent where $\big((\mu_i)_i,\{l^i_1,\cdots ,l^i_{m_i}\}_{i,j},(a^i_1,\cdots,a^i_{n_i})_{i}\big)$ is a basing
for $\Gamma$.
Two peripheral systems for $\Gamma$ are \emph{equivalent}
if they are related by a finite sequence of the following operations:
\begin{itemize}
\item  replacing $\big(G,(\mu_i)_i,\{\lambda^i_j\}_{i,j}, (\alpha^i_j)_{i,j}\big)$ by $\big(G',(\mu'_i)_i,\{{\lambda'_j}^i\}_{i,j}, ({\alpha'_j}^i)_{i,j}\big)$, 
when there exists an  isomorphism $\varphi:G\to G'$, some permutations
  $\sigma_i\in\mathfrak S_{b_i}$ and some elements $w_1,\ldots,w_\ell\in G'$ such that: 
  \begin{itemize}
  \item $\mu'_i=\varphi(\mu_i)^{w_i}$ for all $i\in\{1,\ldots,\ell\}$;
  \item ${\lambda'_{\sigma_i(j)}}^i=\varphi(\lambda^i_{j})^{w_i}$ for all $i\in\{1,\ldots,\ell\}$ and $j\in\{1,\ldots,m_i\}$;
  \item ${\alpha'_j}^i=\overline w_i\varphi(\alpha^i_j)$ for all  $i\in\{1,\ldots,\ell\}$ and $j\in\{1,\ldots,n_i\}$;
  \end{itemize}
\item replacing a loop-longitude by its inverse; 
\item precomposing an $i$--th arc or loop-longitude by some other $i$--th loop-longitude.
\end{itemize}
\end{defi}
We emphasize that peripheral systems depend on a choice of a basing, which corresponds to the choice of a vertex $v_i$ for each connected component $\Gamma_i$, together with a choice of a generating set for $\pi_1(\Gamma_i,v_i)$ and of paths from $v_i$ to each marked vertex.  
But, since $\pi_1(\Gamma_i,v_i)$ is a free group, if follows from Nielsen transformations that equivalence classes are precisely defined to encompass any change of such choices. 
By Proposition \ref{prop:WeldedBasingIso}, we thus have: 
\begin{cor}\label{cor:tex}
Peripheral systems up to equivalence, are well-defined for welded graphs.
\end{cor}
From now on, and by abuse of notation, peripheral systems shall be identified with their equivalence classes.

\begin{remarque}\label{rk:TopologicalCase}
  For a ribbon surface-link $S$, given with a ribbon filling, meridians are
  loops that meet the ribbon filling only once, and longitudes are defined
  by pushing curves on $S$ off the surface. In light of
  the proof of Proposition \ref{prop:TubePi1}, longitudes can
  be characterized by their intersections with the ribbon filling. But
  curves on $S=\Tube(\Gamma)$ induce curves on $\Gamma$, and
  the intersections of the curves push-outs with the ribbon filling, can
  be read out of the curves on $\Gamma$. These precisely
    coincide with the 
    longitudes of $\Gamma$. It
  follows that the Tube map preserves the whole peripheral
  system, in the sense that the peripheral system of a welded graph
  $\Gamma$ is equal  to the peripheral system of $\Tube(\Gamma)$. 
\end{remarque}

\begin{remarque}\label{rk:MWF}
  Suppose that a welded graph has only simply connected components, each having at least one marked vertex. 
  This is for example the case for type $(2,0)$ welded graphs, which
  are intimately related to welded string links by Theorem \ref{th:Embed}.
  In this situation, there is a preferred representative in the equivalence class of peripheral
  systems, corresponding to meridians being the minimal
  marked vertices on each component (see Remark \ref{ohlabelleremarque}). 
  Since there are no loop-longitudes in this case, the peripheral
  system reduces to the data of the associated group, together with
  $\big(\sum_{i=1}^\ell m_i\big)-\ell$ elements, corresponding to paths from the minimal marked vertices to each of other marked vertices on the same component.
\end{remarque}

In light of Remarks \ref{rk:TopologicalCase} and \ref{rk:MWF}, it is clear that, for a string link $L$, the notion of peripheral system for $L$ and the one for $\psi_{SL}(L)$ coincide.
Since by Waldhausen's theorem, peripheral systems are classifying for classical knots, links and
string links; this implies the following result.
\begin{cor}
  Classical string links embed into welded graphs. 
\end{cor}

By Theorem \ref{th:Embed}, this means that, should $\Upsilon$  moves
be non trivial, they would not be able to change the isotopy class of
a classical string link. 
\begin{remarque}\label{rk:links}
For links, the situation is slightly more intricate. 
Indeed, classical peripheral systems provide an orientation for each circular component, whereas welded graph peripheral systems do not. 
For instance, if a link $L_2$ is obtained from another link $L_1$ by mirror image and reversing all orientations, then $\psi_L(L_1)=\psi_L(L_2)$ even if $L_2$ and  $L_1$ are not isotopic. 
Is is thus an interesting question to determine when reversing some longitudes in the peripheral system of a classical
links, corresponds to the peripheral system of a distinct classical link.  
\end{remarque}

\subsection{Reduced quotients and self-virtualization}\label{sec:reduit}

By definition, $G(\Gamma)$ has as many generators as vertices in $\Gamma$. 
Nevertheless, since any two meridians on a same component are conjugate, $G(\Gamma)$ is normally generated by any
choice of one generator per component (as, for example, specified by the choice of a basing). 
In this situation, the following classical notion, due to Milnor \cite{Milnor}, has proved useful: 
\begin{defi}
Let $G$ be a group normally generated by elements $x_1,\cdots,x_p$. We denote by
 $$\RRR G:=\fract{G}{\big\{[x_i,x_i^g]\ \big|\ i\in \{1,\cdots p\}; g\in G\big\}},$$
 the \emph{reduced quotient\index{reduced quotient} of $G$}, which is the largest quotient where each generator commutes with any of its conjugates. 
\end{defi}

This leads to the following definition.
\begin{defi}
  We define a \emph{reduced peripheral system} for a welded graph $\Gamma$ as the data 
  \[
\Big(\RRR G(\Gamma),\big([\mu_i]\big)_i, \big\{[\lambda^i_j].N_i \big\}_{i,j}, \big ([\alpha^i_j].N_i \big)_{i,j}\Big)
\]
where 
\begin{itemize}
\item $\big(G(\Gamma),(\mu_i)_i,\{\lambda^i_j\}_{i,j}, (\alpha^i_j)_{i,j}\big)$ is a peripheral system for $\Gamma$,
\item  $\RRR G(\Gamma)$ is the reduced quotient of $G(\Gamma)$, 
and $[x]$ denotes images in  $\RRR G(\Gamma)$ of $x\in G(\Gamma)$,
\item  $w.N_i$ is the coset of $w$ in $\RRR G(\Gamma)$ relatively to the normal subgroup $N_i$ of $\RRR G(\Gamma)$ generated by
$[\mu_i]$.
\end{itemize}
Two reduced peripheral systems for $\Gamma$ are \emph{equivalent} if they arise from equivalent peripheral systems for $\Gamma$.
\end{defi}

It should be noted that, since two $i$--th meridians are always 
conjugate, the reduced quotient $\RRR G(\Gamma)$ does not depend on
the choice of one meridian per component. Equivalence
classes of reduced peripheral systems are clearly invariant under welded moves. 
\saut

Reduced quotients were introduced by Milnor in \cite{Milnor} for the study of links up to link-homotopies, which are homotopies
that keep distinct components disjoint. In the welded context, the combinatorial counterpart of
link-homotopy is the notion of sv--equivalence,
defined in \cite{ABMW} for welded string links, and that we generalize now to welded graphs:

\begin{defi}
  Two welded graphs differ by a \emph{self-virtualization}\index{self-virtualization} (SV)
  if
  they differ by the replacement of an empty edge $e$ by an edge
  decorated by a letter $v$ or $\overline v$, where $v$ is a vertex
  that belongs to the same component as $e$.
  Two welded graphs are \emph{sv--equivalent} if they are related by a
  sequence of self-virtualizations and welded moves.
\end{defi}

\begin{remarque}\label{rem:sv}
  In terms of Gauss diagrams, an SV move is  the insertion or
deletion of a self-arrow, that is, an arrow which has both endpoints on a same component.  
In terms of virtual diagrams, this translates as replacing a self-crossing by a virtual one, or vice versa. 
\end{remarque}

\begin{prop}\label{prop:RedSelf}
  Reduced peripheral systems are invariant under self-virtualization.
\end{prop}
\begin{proof}
  An empty edge $e$ yields a relation $\overline v_2v_1$ in the group, where $v_1$ and $v_2$ are the two endpoints of $e$. 
  Adding a decoration $v$ on $e$ turns the relation into $\overline
  v_2\overline v v_1 v=\overline v_2[v_1,v]v_1$, which is equal to
  $\overline v_2v_1$  in the reduced quotient since $v$ and $v_2$
  belong to the same component and are hence conjugate one to another (see Remark \ref{rk:Conjugate}). The reduced group is hence invariant.
Moreover, longitudes passing through $e$ are modified from $w_1w_2$ to $w_1vw_2=w_1w_2v^{w_2}$, hence define the same coset.
\end{proof}

\begin{prop}[Chen--Milnor presentation]\label{prop:ChenMilnor}
  Let $\big(G(\Gamma),(\mu_i)_i,\{\lambda_j^i\}_{i,j},(\alpha_j^i)_{i,j}\big)$ be a
  peripheral system for a welded graph $\Gamma$. There exists $w_j^i\in
  F(\mu_1,\ldots,\mu_\ell)$ for
  $i\in\{1,\ldots,\ell\}$ and $j\in\{1,\ldots,b_i\}$ such that
  $w_j^i=[\lambda_j^i]$ in $\RRR G(\Gamma)$ and
  \[
\RRR G(\Gamma)\simeq\RRR\Big\langle\mu_1,\ldots,\mu_\ell\ \Big|\ [\mu_i,w_j^i]\textrm{ for all
    $i\in\{1,\ldots,\ell\}$ and $j\in\{1,\ldots,b_i\}$}\Big\rangle.
    \]
\end{prop}
\begin{proof}
  The core of the proof is to show that $\Gamma$ is sv--equivalent to a welded graph with a unique unmarked vertex (and possibly several marked vertices) per component. This is done component per component. 
 Consider some component $\Gamma_i$, and pick an unmarked vertex: our goal is to eliminate all other unmarked vertices of $\Gamma_i$ by contraction moves. 
  First we remove from the edge labels, all letters corresponding to $i$--th meridians, using Split and self-virtualization moves. 
  Having done so, we can successively apply generalized stabilization moves to those edges that we wish to contract, to turn them into empty edges: 
  since no $i$--th meridian appear in the label that we are sweeping out, those generalized stabilization moves will only modify decorations on
  other components; note that only the labels of other components are affected in the process, so that previously-treated components remain of the desired form. \\
   It remains to observe that, for the resulting welded graph with a unique unmarked vertex per component, the associated Wirtinger presentation provides the desired presentation for $\RRR G(\Gamma)$.
\end{proof}

We stress that the proof of Proposition \ref{prop:ChenMilnor} provides
not only a presentation for $\RRR G(\Gamma)$, but also a ``normal
form''  for a welded graph $\Gamma$ endowed with a basing (hence
with a reduced peripheral system).

Specifically, if $\big((\mu_i)_i,\{l^i_1,\cdots ,l^i_{b_i}\}_{i,j},(a^i_1,\cdots,a^i_{m_i})_{i}\big)$ is a basing for $\Gamma$, 
then the $i$--th component of this normal form has one unmarked vertex $\mu_i$, $b_i$ loops based at $\mu_i$ and decorated by
$l_1^i,\ldots,l_{b_i}^i$, and $m_i$ edges connecting $\mu_i$ to each
marked vertex, and decorated by $a_1^i,\ldots,a_{m_i}^i$:
\[
\dessin{2.25cm}{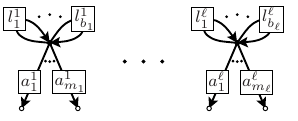}.
  \]
  
This yields the following classification of welded graphs up to self-equivalence by the reduced peripheral system:
\begin{theo}\label{prop:RedClassification}
  Two welded graphs are sv--equivalent if and only if they have
  equivalent reduced peripheral systems.
\end{theo}
\begin{proof}
  The proof essentially follows the proof of \cite[Thm. 2.1]{AM}, which addresses the link case. Consider two welded graphs $\Gamma_1$ and $\Gamma_2$ with equivalent
  reduced peripheral systems, and suppose that $\Gamma_1$
  and $\Gamma_2$ are given in normal forms as above.

  The first step is to modify $\Gamma_2$ into a welded graph having the ``same''  
  reduced peripheral system as $\Gamma_1$;  by this we mean that the isomorphism $\varphi$ of Definition \ref{def:PeripheralSystems} 
  sends directly meridians to meridians and longitudes to longitudes. 
  As stated in this same definition, equivalence of reduced peripheral systems is generated by three elementary operations. 
  The first one corresponds to conjugation, by a same word $w$, of a meridian and the associated loop-longitudes, and precomposition by $\overline w$ of the associated arc-longitudes; at the level of $w$--graphs, this can be 
  realized by generalized stabilizations. 
  The second operation reverses a loop-longitude; this can be realized by an orientation reversal move on $w$--graphs. 
  The last operation precomposes a longitude by a loop-longitude from the same component; this can be realized on $w$--graphs by the following sequence of moves:
  \[
      \dessin{1.75cm}{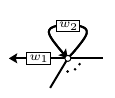} \ \xrightarrow[]{\ \textrm{E}}\
      \ 
      \dessin{1.75cm}{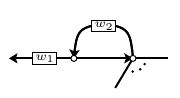} \ \xrightarrow[]{\ \textrm{P}\
                                         } \
                                             \dessin{1.75cm}{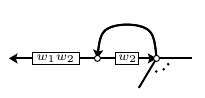}\
                                             \rotatebox{30}{$\xrightarrow[]{\
                                               \textrm{C}\ }$}
                                             \begin{array}{c}
                                               \dessin{1.75cm}{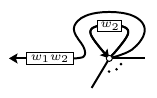}\\
                                               \rotatebox{270}{$\simeq$}\\
                                       \dessin{1.75cm}{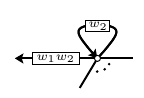}
                                             \end{array}.
    \]
Now, $\Gamma_1$ and $\Gamma_2$ only differ by the representative words
for the longitudes. Thanks to Proposition \ref{prop:ChenMilnor}, we know
that, for a given $i$--th longitude, the representative words differ by a sequence of
insertion/deletion of \begin{itemize}
\item conjugates of $\mu_i$,
\item  commutators $[\mu_j,w_j]$ for some $j\in\{1,\ldots,\ell\}$, where $w_j$ is a $j$--th loop-longitude,
\item commutators $[\mu_j,\mu_j^w]$ for some $j\in\{1,\ldots,\ell\}$, where $w$ is any word.
\end{itemize}
All these operations are handled in a local way in the proof of
\cite[Thm. 2.1]{AM}, and the argument can be transposed directly to the welded graph case. 
\end{proof}

\begin{remarque}\label{rem:jbalancetout}
  Theorem \ref{prop:RedClassification} can be seen as a
  generalization of the main theorem of \cite{AM}, which deals with the link case. 
  It provides an equivalence between reduced peripheral systems and welded graphs up
  to sv--equivalence. 
  Other quotients of the peripheral system, such
  as $q$--nilpotent \cite[Thm. 2]{Colombari} or $k$--reduced \cite[Thm. 4.3]{kAMY} quotients  have
  also been investigated and have led to similar correspondances for welded links
  up to,  respectively,  $w_q$ and self $w_k$--concordance. 
  The notion of (self) $w_q$--concordance refers, on the one hand, to a combinatorial notion of concordance for welded objects which generalizes the classical notion of concordance for topological objects and, on the other hand,
the notion of (self) $w_q$--equivalence based on the \emph{arrow
  calculus} developed by the second and third authors \cite{arrow}. This arrow calculus 
relies on surgery operations on welded diagrams, along certain oriented unitrivalent trees, 
filtered by the size of the associated tree.
\\
It turns out that arrow calculus can be translated in terms of welded graph. A degree $q$ surgery corresponds to inserting a
length $q$ iterated commutator in some edge decoration, and a degree $k$ self-surgery likewise corresponds to inserting an iterated commutator which involves at least $k+1$ times the same component. The diagrammatic formulation of concordance can also be extended to all welded graphs in a straightforward way. 
All the arguments extend to this setting, so that two welded graphs are
$w_q$--concordant (resp. self $w_k$--concordant) if and only if they have
equivalent $q$--nilpotent (resp. $k$--reduced) peripheral systems.
\\
  Note that, all together, this implies that, up to sv--equivalence, $w_q$ or self
  $w_k$--concordance, the sets of welded links and string links embed in
  welded graphs. As a matter of fact, and compared with
  Theorem \ref{th:Embed}, it can indeed be directly
  checked by hand that $\Upsilon$ moves are trivial up to
  sv--equivalence, $w_q$-equivalence or self $w_k$--equivalence. 
\end{remarque}

\subsection{Milnor invariants}

The peripheral system and its reduced quotient are strong invariants, but they are also very difficult to compute and compare. 
In \cite{Milnor2}, Milnor introduced numerical invariants extracted from the nilpotent peripheral system, which are
algorithmically computable, and which are still powerful link invariants.  We define below similar invariants for welded graphs. 
A comprehensive account of Milnor invariants can be found in \cite[Sec. 4]{cutAMY}, see also the survey \cite{surveyMilnor}.

For simplicity of exposition, here, we will restrict ourselves to the so-called ``non repeated''  Milnor invariants, that are extracted from the reduced peripheral system. 
Moreover, with our main topological result in sight (see Section \ref{sec:Topology}), we shall further restrict to the following class of welded graphs. 
\begin{defi}
A  \emph{welded forest}, is a welded graph of type $\big((m_1,0),\ldots,(m_\ell,0)\big)$ such that  $m_1,\ldots,m_\ell\neq0$. 
In other words, each connected component is a rooted tree:
  simply connected and with a minimal marked vertex.
\end{defi}

Let $\Gamma$ be a welded forest. 
Following Remark \ref{ohlabelleremarque}, $\Gamma$ has a canonical set of meridians $\mu_1,\ldots,\mu_\ell$, 
which corresponds to the minimal marked vertex on each component. 
Moreover, $\Gamma$ has no non-trivial loop-longitude, so that by
Proposition \ref{prop:ChenMilnor}, we have that $\RRR G(\Gamma)$ is isomorphic to the reduced free group $\RRR  F(\mu_1,\ldots,\mu_\ell)$.
Hence, following Remark \ref{rk:MWF}, the reduced peripheral system of $\Gamma$ is simply the data 
of $\big(\sum_{i=1}^\ell m_i\big)-\ell$ elements of $\RRR F(\mu_1,\ldots,\mu_\ell)$, representing all arc-longitudes. 
In what follows, we shall simply denote by $\alpha^i_j\in F(\mu_1,\ldots,\mu_\ell)$, the element representing the arc-longitude running to the $j$--th marked point on the $i$--th component of $\Gamma$ ($i\in\{1,\cdots,\ell\}$, $j\in\{2,\cdots,n_i\}$).

Next we recall the \emph{reduced Magnus expansion}
\[
\begin{array}{ccc}
  \RRR
F(\mu_1,\ldots,\mu_\ell)&\to&\fract{\Z\langle\!\langle
                                           X_1,\ldots,X_\ell\rangle\!\rangle}{\RR}\\
  \mu^\e_i&\mapsto&1+\e X_i
\end{array},
\]
taking values in the quotient of the ring $\Z\langle\!\langle
X_1,\ldots,X_\ell\rangle\!\rangle$ of formal power  series in $\ell$
non-commuting variables, by the ideal $\RR$  generated by monomials containing at least twice the same variable.
The reduced Magnus expansion is known to be injective, see for instance \cite{ipipipyura}.

\begin{defi}
 For every $i\in\{1,\ldots,\ell\}$, $j\in\{2,\ldots,m_i\}$, and for every sequence $I=i_1\ldots i_r$ of distinct integer in
 $\{1,\ldots,\ell\}\setminus\{i\}$, we define $\mu_\Gamma(I;i,j)$ as
 the coefficient of $X_{i_1}\cdots X_{i_r}$ in $E(\alpha_j^i)$. 
 These coefficients are called  the \emph{non repeated Milnor
   invariants}\index{Milnor invariants} of $\Gamma$.
\end{defi}

 \begin{remarque}
 The well-definedness of these non repeated Milnor invariants $\mu_\Gamma(I;i,j)$ is clear. 
 Indeed, as discussed above, the welded forest $\Gamma$ has a preferred set of meridians (Remark \ref{rk:MWF}), 
 and  the $\big(\sum_{i=1}^\ell m_i\big)-\ell$ arc-longitudes forming its reduced peripheral system, 
 induce well-defined elements $\alpha^i_j$ in  $\RRR G(\Gamma)= \RRR
 F(\mu_1,\ldots,\mu_\ell)$ (see Notation \ref{cvraimentjusteunenota}).
 Their invariance under self-virtualization follows from Proposition \ref{prop:RedSelf}. 
 \end{remarque}

We have the following classification result.
\begin{prop}\label{prop:wTreeClassification}
  Two welded forests are sv--equivalent if and only they have
  same non repeated Milnor invariants.
\end{prop}
\begin{proof}
This is a direct consequence of Theorem \ref{prop:RedClassification},
combined with the injectivity of the reduced Magnus expansion  \cite{ipipipyura}.
\end{proof}

\section{Application to knotted punctured spheres up to link-homotopy}
\label{sec:Topology}

We can now get back to 4--dimensional topology, and deal with the case of knotted punctured spheres. 
\begin{defi}
Let $\ell$ be some positive integer and $(m_1,\ldots,m_\ell)\in\N$. A
union of $\ell$ \emph{knotted punctured spheres}\index{knotted
  punctures spheres}
is the image of a proper embedding in $B^4$ of
$\Sigma=\sqcup_{i=1}^\ell S_{\! i}$, where $S_{\! i}$ is a $2$--sphere with $m_i$ disjoint disks removed and such that
$\partial\Sigma$ is sent to a fixed (oriented and ordered) unlink $U$
in $\partial B^4=S^3$, up to isotopy fixing the boundary ; we say
it has type $(m_1,\ldots,m_\ell)$.
\end{defi}
As noted in the introduction, this large class of surface-links contains in particular Le Dimet's linked disks \cite{LeDimet} (case $m_i=1$ for all $i$), and $2$--string links \cite{ABMW,AMW} (case $m_i=2$ for all $i$).
We stress that the boundary components of each punctured sphere are ordered, so that there is a preferred meridian for each component of a union $S$ of knotted punctured spheres, defined as a loop in $\partial B^4\setminus S$ which is a meridian of the first boundary component. 

Our purpose is to classify unions of knotted punctured spheres up to link-homotopy. 
\begin{defi}
Two unions of knotted punctured spheres are \emph{link-homotopic} if they are
images of embeddings which are homotopic through proper immersions
that keep distinct components disjoint, and that send $\partial\Sigma$ to $U$. 
\end{defi} 

Our first result in this direction, is a generalization of \cite[Thm. 3.5]{AMW}.
\begin{prop}\label{prop:Ribbon}
  Any union of knotted punctured spheres is link-homotopic to a ribbon one.
\end{prop}
\begin{proof}
We can use the very same strategy as used in \cite[Sec. 3.2]{AMW} in the $2$--string link case. 
Pick some union $S$ of knotted punctured spheres. 
The strategy goes in three steps, which can be roughly summarized as follows. 
First, using the fact that $\partial S$ is an unlink, one can shrink and stretch a neighborhood of this boundary; 
this allows us  to regard $S$ as a $2$--link (\emph{i.e.} smoothly embedded disjoint copies of the $2$--sphere) with thin, unknotted and unlinked tubes attached. 
Next, by a theorem of Bartels and Teichner \cite[Thm.~1]{BT}, there is a link-homotopy which takes this $2$--link to a disjoint union of
unknotted spheres.
The proof then amounts to showing that, in the process of this link-homotopy, only ribbon-type linking  will be created among the attached tubes, so that the final 
union of knotted punctured spheres is indeed ribbon. 
This last step uses broken surface diagrams of immersed surfaces, and consists in a systematic study of all Roseman moves and  singular Roseman moves on these diagrams, which are local moves that generate link-homotopy of surface-links \cite[Prop.~2.4]{AMW}; the arguments in \cite{AMW} being local, they apply verbatim to our more general situation. 
\end{proof}

In \cite[Sec. 5.2.3]{cutAMY}, Milnor link-homotopy invariants for surface-links were defined; 
they come in three flavors: Milnor maps, Milnor loop-invariants and Milnor arc-invariants. 
Milnor maps and Milnor loop-invariants are extracted from the loop-longitudes; 
in the case of a union $S$ of $\ell$ knotted punctured spheres, 
the latter are given by curves parallel to the boundary components, which are copies of the unknot, 
so that Milnor maps and Milnor loop-invariants of $S$ are trivial. 
Milnor arc-invariants are defined by pushing in $B^4\setminus S$, and closing in a canonical way,
the $\big(\sum_{i=1}^\ell m_i\big)-\ell$ boundary-to-boundary arcs given by the basing. 
These define a collection of elements of the reduced fundamental group of $B^4\setminus S$, which 
by \cite[Thm.~5.15]{AMW} is 
isomorphic to the reduced free group generated by $\ell$ elements $m_1,\cdots,m_\ell$, 
represented by the preferred meridians of each component of $S$.
Taking the coefficients in the reduced Magnus expansion of these elements, defines non-repeated Milnor invariants of $S$.  
We refer to \cite{cutAMY} for details.

We can now prove the classification result announced in \cite[Thm.~6.10]{cutAMY}.
\begin{theo}\label{thm:classif}
  Unions of knotted punctured spheres are classified, up to link-homotopy, by non repeated Milnor invariants.
\end{theo}
\begin{proof}
    By \cite[Prop. 5.23]{cutAMY}, non repeated Milnor invariants are invariant under link-homotopy.
    Now, let $S_{\! 1}$ and $S_{\! 2}$ be two unions of knotted punctured spheres with the same non repeated Milnor invariants. 
    By Proposition \ref{prop:Ribbon}, we may assume that $S_{\! 1}$ and $S_{\! 2}$ are ribbon, so that 
    by Proposition \ref{prop:Tube}, they can be described as 
    $\Tube(\Gamma_1)$ and $\Tube(\Gamma_2)$ for two welded forests $\Gamma_1$ and $\Gamma_2$. 
    It follows from Remark \ref{rk:TopologicalCase} that the non-repeated Milnor arc-invariants of $S_{\! i}$
    coincide with those of $\Gamma_i$ ($i=1,2$).
    Hence $\Gamma_1$ and $\Gamma_2$ have same non-repeated Milnor invariants, and 
    Proposition \ref{prop:wTreeClassification} implies that $\Gamma_1$ and $\Gamma_2$ are therefore sv--equivalent. 
    Now, in the same way as it was observed 
    in \cite{ABMW} in the welded string link case, an SV move on a welded graph can be realized as a link-homotopy on its image under $\Tube$; 
    more precisely, the image by Tube of an SV move is a ``self-circle
    crossing change'' as illustrated in \cite[Fig.~6]{ABMW}.
    Thus $S_{\! 1}$ and $S_{\! 2}$ are link-homotopic, and the proof is complete. 
\end{proof}

Note that, in particular, Theorem \ref{thm:classif} implies that linked disks (that is, unions of once-punctured knotted spheres) are all link-homotopically trivial, 
since linked disks admit no non-trivial longitudes. 
The following consequence, using results of \cite{MY7}, might be well-known to the experts.
\begin{cor}\label{cor:slice}
 Up to link-homotopy, a slice link bounds a unique union of slice disks. 
\end{cor}
\begin{proof}
Let $L$ be a slice link in $S^3$, and let $D_1$ and $D_2$ be two unions of slice disks for $L$ in the $4$--ball. 
Pick in $S^3\times [0,1]$ a concordance $C$ from $L$ to the unlink $U$, 
and let $\overline{C}$ be the mirror image of $C$.
By \cite[Lemma 1.1]{MY7}, the stacking product $C\cdot \overline{C}$ is link-homotopic to the product $L\times [0,1]$. 
Since, for $i=1,2$,  a collar neighborhood of the boundary of $D_i$ is isotopic to $L\times [0,1]$, we have that $D_i$ is link-homotopic to a product $C\cdot S_{\! i}$, 
where $S_{\! i}$ is a union of slice disks for the unlink $U$. 
But as noted above, by Theorem \ref{thm:classif} we have that $S_{\! 1}$ and $S_{\! 2}$ are link-homotopic. 
It follows that $D_1$ and $D_2$ are link-homotopic, as desired. 
\end{proof}
\begin{remarque}\label{rem:ark}
The trick used in the proof of Corollary \ref{cor:slice} above,
relying on \cite[Lemma~1.1]{MY7}, applies more generally to all knotted punctured spheres. 
This allows to extend the classification of Theorem \ref{thm:classif} to a larger class of surface-links, which are  images of  proper embeddings in $B^4$ of
a union of punctured spheres $\sqcup_{i=1}^\ell S_{\! i}$, such that
the boundary is sent to a fixed \emph{slice} link in $\partial B^4=S^3$. 
\end{remarque}

For the classification given in Theorem \ref{thm:classif}, we considered link-homotopies sending the boundary to the fixed unlink $U$ at all time. 
This condition can be relaxed to define \emph{weak link-homotopies}, as homotopies through proper immersions that may
not preserve the image of the boundary.
\begin{cor}
 Any two unions of knotted punctured spheres with same type are weakly link-homotopic.
\end{cor}
\begin{proof}
Let us call \emph{union of unknotted punctured spheres}, a union of punctured spheres which admits a ribbon filling with no ribbon disk. It follows
from the same arguments as in the proof of Lemma \ref{lem:CBij}, that such a union of unknotted punctured spheres is unique. 
We denote by $S_{\! 0}$ these unknotted punctured spheres. 
Let us prove that any union $S$ of $\ell$ knotted punctured spheres is weakly link-homotopic to $S_{\! 0}$.

  The strategy is to provide a union of $\ell$ knotted punctured
  spheres $S'$ which has, on one hand, the same non repeated Milnor invariants as
  $S$, so that it is link-homotopic to $S$ by 
Theorem \ref{thm:classif}, and which is on the other hand ``nice enough''  to be unknotted by an ambient isotopy that does not fix $\partial B^4$.
For that, we consider the arc-longitudes of $S$, seen as words in 
the reduced free group on $\ell$ generators, and build an $\ell$--component welded string link $L$ realizing theses words as its longitudes: see for example \cite[Fig.~2.1]{MY7} or \cite[Rem.4.23]{ABMW}. 
The image of $L$ by the Tube map is an $\big(\sum_{i=1}^\ell m_i\big)$--component $2$--string link which is ribbon.
By \cite[Theorem 2.30]{ABMW}, $\Tube(L)$ is link-homotopic to a \emph{monotone} ribbon $2$--string link, the latter being the $4$--dimensional trace of an ambient isotopy $(f_t)_{t\in[0,1]}$ of an unlink in $S^3$.
On the unlink $U$ forming the ``upper'' boundary of this monotone
ribbon $2$--string link, we glue a copy of $S_{\! 0}$: this yields the
desired union of knotted punctured spheres $S'$. By construction, it
has the same Milnor invariants as $S$, and performing backward
$(f_t)_{t\in[0,1]}$ on a collar neigborhood of $\partial B^4$ does
 provide an ambient isotopy which unknot $S'$ into $S_{\! 0}$.
\end{proof}

\section{Proof of Theorem \ref{th:Embed}}\label{sec:delamort}

We prepare below three technical lemmas that will be used in the proof
of Theorem \ref{th:Embed}. 

\begin{lemme}\label{lem:Star}
  Two Gauss diagrams which differ only locally
  as
  \[
\dessin{1.875cm}{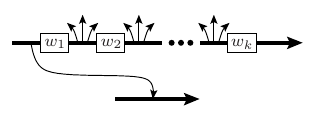}\hspace{1cm}\textrm{and}\hspace{1cm}\dessin{1.875cm}{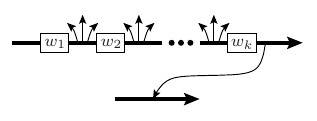},
\]
where each triad of open arrows stands for any bunch of tails, are equivalent up to welded moves, if the concatenated word $w_1\cdots w_k$ is trivial as a free group element.
\end{lemme}
\begin{proof}
  The leftmost tail in the left-hand diagram, can be pushed to the
  right at the cost of conjugating its head by the word $w_1\cdots
  w_k$ by the exchange moves of Figure \ref{fig:exchange}. But since
  this word is trivial as a
  free group element, all the conjugating arrows can be removed pairwise using Reid2 moves.
\end{proof}

\begin{lemme}\label{lem:GenHH_1}
  Up to welded moves, two Gauss diagrams which differ only locally as
  \[
    \dessin{2.97cm}{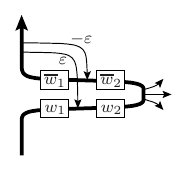}
    \hspace{1cm}\textrm{and}\hspace{1cm}
        \dessin{2.97cm}{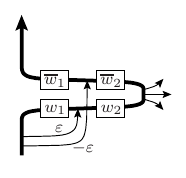},
      \]
      where none of the right-pointing tails are connected to $w_1$,
  are equivalent.
\end{lemme}
\begin{proof}
  We start from the left-hand diagram. We denote by $a_{\varepsilon}$ and $a_{-\varepsilon}$ the depicted $\varepsilon$ and
  $(-\varepsilon)$--signed arrows, and by $x_t$ and $x_b$ the depicted vertical top and bottom sub-arcs of the diagram. 
  Then $\overline w_1$ can be pushed through the head of $a_{-\varepsilon}$, so that
  $a_{-\varepsilon}$ can be removed with a Reid1 move.
  Pushing $\overline w_1$ through  $a_{-\varepsilon}$ in this way,
  this has the effect of conjugating all tails of arrows in $\overline w_1$ by $x_t$, and we denote this by replacing the label $\overline w_1$ by $\overline w_1^{x_t}$. 
Next, $w_1$ can be pushed through the head of 
  $a_{\varepsilon}$, at the cost of conjugating  $w_1$ by $x_t$.\footnote{Conjugation is also by $x_t$
    since $a_{\varepsilon}$ and $a_{-\varepsilon}$ have opposite signs
  and we are pushing $w_1$ and $\overline w_1$ in opposite directions according to the diagram orientation.} 
  We thus obtain the local diagram 
\[
\dessin{2.34cm}{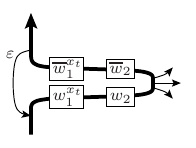}. 
\]
Now, observe that each letter in $w_1$ corresponds to two arrows, one from $w_1$ and one from $\overline w_1$, which can be assumed to have  adjacent tails by move TC. When conjugated by $x_t$, these tails are now separated by two $x_t$--tailed arrows with opposite signs which can be removed using a Reid2 move. 
By this observation, we have that $w_1^{x_t}$ and $\overline w_1^{x_t}$ indeed represent inverse elements in the free group: we are therefore in the situation of Lemma \ref{lem:Star}. 
We apply hence Lemma \ref{lem:Star}, not only to the tail of $a_\varepsilon$, but also to all $x_t$--tailed arrows conjugating $w_1^{x_t}$ and $\overline w_1^{x_t}$: all these tails are pulled in this way down to the bottom interval $x_b$. After this deformation, the arrow $a_\varepsilon$ can be deleted thanks to a Reid1
move, while all conjugating $x_t$--tailed arrows are turned into
conjugating $x_b$--tailed arrows, since their tails now sit on $x_b$. As a result, we obtain the following local diagram: 
\[
\dessin{2.19cm}{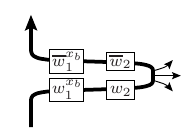}.
\]

We now turn to the right-hand diagram in the statement. 
We can perform a similar sequence of moves as above, except that, at the final step, we only apply Lemma \ref{lem:Star} to the arrow $a_{-\varepsilon}$ (that is, we leave all conjugating $x_b$--tailed arrows 
on $x_b$). This yields the very same diagram as above, which concludes the proof. 
\end{proof}

\begin{lemme}\label{lem:GenHH_2}
  Up to $\Upsilon$ and welded moves, two Gauss diagrams which
  differ only locally as:
     \[
    \dessin{2.97cm}{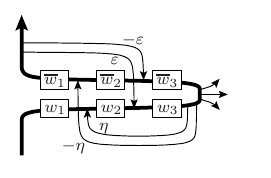}
    \hspace{1cm}\textrm{and}\hspace{1cm}
        \dessin{2.97cm}{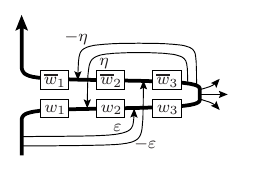},
      \]
      where none of the right-pointing tails are connected to $w_1$,
      $w_2$ or $w_3$,
  are equivalent.
\end{lemme}
\begin{proof}
  We denote by $a_\alpha$
  the depicted $\alpha$--signed arrow (for
  $\alpha\in\{\pm\varepsilon,\pm\eta\}$), and by $x_t$ and $x_b$ the
  depicted vertical sub-arcs of the Gauss diagram. We also denote by $x_r$ the rightmost sub-arc of the diagram, which contains the tails of arrows $a_\eta$ and $a_{-\eta}$.

  We first note that in the case where $w_2=w_3=\emptyset$, the local deformation of the  statement is nothing but the $\Upsilon$ move; hence the case  $w_2=w_3=\emptyset$ is immediate.
  
  Now let us show how the case where $w_3=\emptyset$, reduces to this  $w_2=w_3=\emptyset$ case. 
  This is done by pushing, in the left-hand side diagram, $w_2$ and  $\overline w_2$ across the head of $a_\eta$ and $a_{-\eta}$, respectively. 
  As in the proof of Lemma \ref{lem:GenHH_1}, this deformation is  achieved at the cost of conjugating these words $w_2$ and
  $\overline w_2$ by $x_r$. 
 We can then apply the $w_2=w_3=\emptyset$ case, and then push back the $x_r$--conjugated word $w_2$, resp.  $\overline w_2$, across the head of $a_\eta$, resp. $a_{-\eta}$. Being the reverse process of the initial deformation, this simply unconjugates these words, and yields the right-hand-side diagram.
 
 Finally, we observe that the general case similarly reduces to the $w_3=\emptyset$ case. 
 The argument is the same as above, except for one additional subtlety. 
Starting with the left-hand diagram, we push $w_3$ and  $\overline w_3$ across the head of $a_\varepsilon$ and $a_{-\varepsilon}$, respectively. 
We can then apply the $w_3=\emptyset$ case of the statement. 
The point is that the initial deformation was achieved at the cost of conjugating the words $w_3$ and $\overline w_3$ by $x_t$ (and not $x_r$ as above). Hence we cannot immediately push back these conjugated words to obtain the desired move. 
But each of the $x_t$--conjugating arrows introduced at this initial
stage, has its tail on the top vertical sub-arc $x_t$, and, by assumption, its head is located away from the depicted part of
the diagram. Hence Lemma \ref{lem:Star} can be applied to each of
these arrows, whose tail can be moved down to the bottom vertical part
$x_b$. Indeed, the concatenation of the words separating $x_t$ and $x_b$ is equivalent to the empty word, so that the lemma applies.
After applying Lemma \ref{lem:Star} in this way, $w_3$ and $\overline w_3$ can be unconjugated by pushing them back in place, and the proof is complete.
\end{proof}

Before proceeding with the proof of Theorem \ref{th:Embed}, we introduce some terminology.

Each connected component of a $w$--graph of type $(2,0)$ has two marked vertices, which are ordered according to the orientation: 
the first one is called \emph{initial vertex} and the second one  is called \emph{final vertex}. There is also a unique shortest path from the initial to the final vertex: every edge which is not on this path is called an \emph{extra edge}. 
Note that the $w$--graph is linear if it has no extra edge. 
\\
Every vertex $v$ of a $w$--graph of type $(2,0)$ which is not the final vertex, has a unique shortest path to the final vertex: the first edge on this path is called the \emph{out edge} of $v$. For vertices distinct from the initial vertex, the \emph{in edge} is similarly defined as the first
edge on the shortest path to the initial vertex. 
Note that the in and out edges of a vertex are distinct if and only if the vertex is on the shortest path from the initial to the final vertex.

In the same way, each connected component of a $w$--graph of type $(0,1)$ has a unique shortest non contractible path: edges which are not on this
path are also called \emph{extra} edges. Observe that the $w$--graph is cyclic if it has no extra edge. 

\begin{proof}[Proof of Theorem \ref{th:Embed}]
  Let us  consider the string link case.
  First, we observe that the map $\psi_{SL}$ is invariant under $\Upsilon$ moves.  
  Indeed, on the one hand we have by definition that 
  \[\psi_{SL}\left(\dessin{2.275cm}{HH_3}\right)=\dessin{2.5cm}{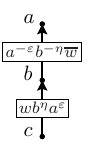}, 
  \]
  and applying a sequence of expansion, contraction and push moves to the resulting w-graph,  shows that 
  \[
   \dessin{2.5cm}{IInvHH_1} \xrightarrow{E}\dessin{2.5cm}{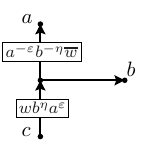}\xrightarrow{P}\dessin{2.5cm}{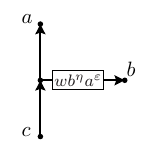}\xrightarrow{C} \dessin{2.5cm}{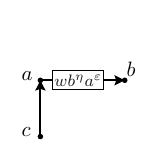}\xrightarrow{C}\dessin{2.5cm}{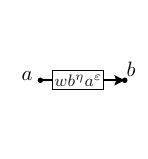},
  \]
 On the other hand, applying the $\psi_{SL}$ map and a similar
 sequence of expansion, contraction and push moves, gives: 
  \[
\psi_{SL}\left(\dessin{2.275cm}{HH_4}\right)=\dessin{2.5cm}{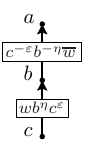}\xrightarrow{E,P,C}\dessin{2.5cm}{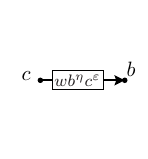}=\dessin{2.5cm}{IInvHH_3}.
\]
Therefore $\psi_{SL}$ induces a well-defined map 
\begin{gather*}
  \overline \psi_{SL}: \big\{\textrm{welded string links up to $\Upsilon$ moves}\big\}\to \big\{\textrm{welded graphs of type $(2,0)$}\big\}.
\end{gather*}

We now define an inverse map $\overline\xi_{SL}$ for $\overline\psi_{SL}$. 
Let $\Gamma$ be a $w$--graph of type $(2,0)$. 
The description of $\overline\xi_{SL}(\Gamma)$ will be given in three steps, summarized in Figure \ref{fig:Linearization}. 
The first step provides an algorithm for turning $\Gamma$ into a linear $w$--graph, by a canonical sequence of push and
contraction moves; the second step specifies the edge labels on this linear $w$--graph, induced by those of $\Gamma$; 
the third step associates a  welded string link to the resulting welded graph.
\begin{figure}
  \[
    \begin{array}{ccccc}
    \dessin{2.5cm}{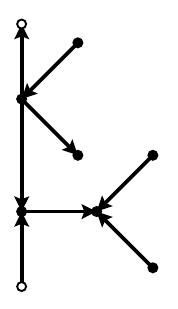} \longrightarrow
      \dessin{2.5cm}{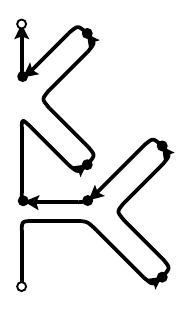}&&\dessin{2.5cm}{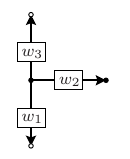} \longrightarrow
      \dessin{2.5cm}{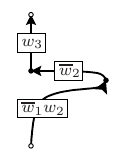}&&
                                     \begin{array}{c}
\dessin{.625cm}{Word_1}\\\rotatebox {270}{$\leadsto$}\\\dessin{1.41cm}{Word_2}
                                     \end{array}
\\
      \textrm{step 1}&\hspace{1cm}&\textrm{step 2}&\hspace{1cm}&\textrm{step 3}
    \end{array}
  \]
  \caption{Definition of $\overline\xi_{SL}$}
  \label{fig:Linearization}
\end{figure}

\emph{Step 1.} For every vertex of $\Gamma$, we fix an arbitrary
ordering on the adjacent edges such that the out edge is last and,
if it exists, the in edge is penultimate. The first edge of a vertex
is called the \emph{leading} edge of the vertex, and the second one is
called the \emph{co-leading} edge. Using these orderings clock-wisely,
we obtain a planar way of embedding $\Gamma$ in $\R^2$; 
we can then ``contour''  these planar trees (see step 1 in Figure
\ref{fig:Linearization}) to get a union $I$ of intervals,
connecting initial vertices to final ones. We orient these intervals 
from the initial to the final vertex. Each vertex of $\Gamma$ corresponds
to at least one, possibly more, points on $I$: we keep only the point which is closest
to the final vertex. This turns $I$ into a linear graph, whose vertices are in one-to-one correspondence with
those of $\Gamma$, and whose edges are unions of edges of $\Gamma$. 

\emph{Step 2.}  Each edge $e$ of $I$ is  a union of edges $e_1,\cdots, e_k$  of $\Gamma$ with the same
or the reversed orientation. Let $w_i$ be the decoration of $e_i$ in $\Gamma$. Then we
decorate $e$ by the concatenation $w_1^{\e_1}\cdots w_k^{\e_k}$ where $e_i\in\{\pm 1\}$ is $1$ if and only if
the orientations of $e$ and $e_i$ agree; see step 2 in Figure \ref{fig:Linearization} for an
illustration. 
\\
At this point, we obtain a linear $w$--graph $\Gamma_l$ which is equivalent to $\Gamma$. Indeed, 
$\Gamma_l$ is the result of iteratively contracting, as outlined in the proof of Lemma \ref{lem:PC}, all extra edges with a univalent unmarked endpoint, starting with the one which is closest to the final vertex of the corresponding component of $I$. 

\emph{Step 3.}  
We apply, on the linear $w$--graph $\Gamma_l$, the exact same procedure as in the latter half of the proof of Lemma \ref{lem:PC}:  
we split each edge by S moves to obtain a linear $w$--graph labeled by
letters, and we replace each of these letter labels by an arrow as
illustrated in step 3 of Figure \ref{fig:Linearization}. The result is the desired welded string link $\overline\xi_{SL}(\Gamma)$; as observed in the proof of Lemma \ref{lem:PC}, 
this is indeed a preimage for $\Gamma_l$ by the map $\psi_{SL}$. 
\saut
 
We now have to prove that $\overline\xi_{SL}(\Gamma)$ is well defined, that is, $\overline\xi_{SL}(\Gamma)$ is invariant under all moves on $w$--graphs, and does not depend on the choice of planar embedding for $\Gamma$ (that is, on the chosen orderings at each vertex). 

First, it is easily checked that OR moves do not change the resulting Gauss diagram, so that we can freely assume that all edges are oriented toward the final vertex. 
Invariance under $S$ moves (hence under push moves) is also easy to check, as they introduce pairs of arrows which can be deleted by R2 moves.

Contraction moves turn out to be the most intricate moves to check. 
Contraction/expansion of a leading edge, which we shall call \emph{leading contraction/expansion} below, has no impact, so that invariance is directly checked. 
But contraction of  a non leading edge may change  drastically the position of some vertex in $\Gamma_l$. \\
First, let us observe that contraction of an arbitrary (non leading)
edge $e$ can always be realized by a sequence of expansion and contractions on either leading edges, or co-leading edges at a trivalent vertex; this is illustrated in the following figure, where $T$ is any tree attached to the edge $e$:\\[-0.4cm] 
\[
\dessin{1.25cm}{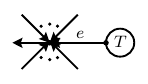}\rightarrow
\dessin{1.72cm}{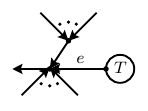}\rightarrow
                          \dessin{1.72cm}{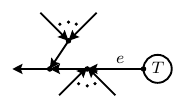}\rightarrow
                          \dessin{1.72cm}{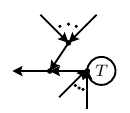}
\rightarrow \dessin{1.72cm}{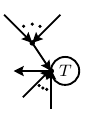}\rightarrow \dessin{1.25cm}{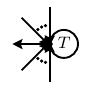}.
  \]
  There, the first step is a leading expansion,  followed by an expansion 
  on a  co-leading edge at a trivalent vertex; this turns edge $e$ into a leading edge, so that contraction can be performed; it then only remains to perform the inverse of the first two operations. \\
Hence we are left with proving the case of a contraction of a co-leading edge $e$ at a trivalent vertex. 
Using push moves and  leading contractions, the tree attached via the leading edge can be assumed to be linear, that is, we consider a $w$--graph of the form:\\[-0.5cm]
 \[
\dessin{1.56cm}{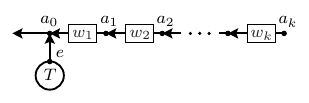},
\]
where as above, $T$ is any tree, attached to the edge $e$ that we would like to contract. 
By recursively applying push moves and leading expansions, this $w$--graph can be turned into the following: \\[-0.5cm]
 \[
\dessin{2.2cm}{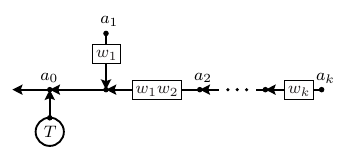} \leadsto \cdots \leadsto  \dessin{2.2cm}{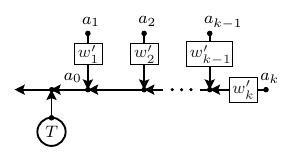} \leadsto  \dessin{2.2cm}{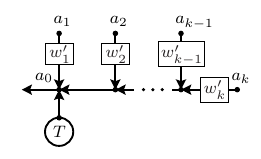}, 
\]
where the final step is achieved by a leading contraction. 
Furthermore, we may freely assume that for each $i=1,\cdots,k$, the letter $a_i$ appears at most once in the word $w'_i$. 
This is done by using iteratively the following trick during the above process:\\[-0.75cm] 
\[
\dessin{1.72cm}{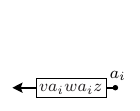}\rightarrow\dessin{1.72cm}{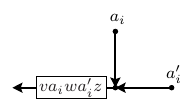}\rightarrow\dessin{1.72cm}{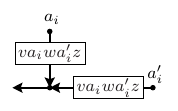}.
  \]
  Here, the first operation is a leading expansion  (which allows for replacing each label $a_i$ by either $a_i$ or $a'_i$), followed by a push move. \\
Now, we show the invariance under the contraction of $e$  by induction on $k$. 
For $k=1$, there are two cases:
\begin{itemize}
\item Suppose that $a_1$ does not appear in $w_1$. 
 The figure below gives the image by $\overline\xi_{SL}$ of the $w$--graph before (left) and after (right) contraction on the edge $e$: 
  \[
    \dessin{2.5cm}{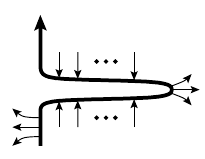}
    \hspace{1cm}\textrm{and}\hspace{1cm}\dessin{2.5cm}{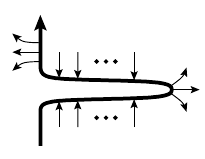},
  \]
  where each pair of opposed vertical arrows have adjacent tails and opposite signs. 
  Note that the bunch of right-pointing arrows represent all occurrences of $a_1$ in the edge labelings, 
  so that we know by assumption that none of these tails is connected to the vertical arrow heads of the figure. 
  Starting from the left-hand side figure above, we aim at moving all left-pointing arrow tails from the bottom to the top vertical strand. 
  This is achieved by Lemma \ref{lem:Star} for any such arrow that is not connected to the vertical arrow heads; 
  left-pointing arrows that are connected to the vertical arrow heads, come by pairs with opposite signs, so that Lemma \ref{lem:GenHH_1} can be applied to perform the desired move.
\item Suppose that $a_1$ appears once in $w_1$. 
  Concretely this means that, in the above left-hand figure, a pair of
  vertical arrows corresponds to two of the right-pointing arrows:
    \[
    \dessin{2.5cm}{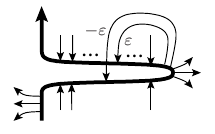}.
  \]
  This case is handled as the previous one, except that we might use Lemma \ref{lem:GenHH_2} instead of Lemma \ref{lem:GenHH_1}.
\end{itemize}
The inductive step is illustrated below:
\[
  \dessin{2.5cm}{GenHH_3}\hspace{.5cm} \raisebox{.75cm}{\rotatebox{30}{$\rightarrow$}}
  \begin{array}{c}
    \dessin{2.5cm}{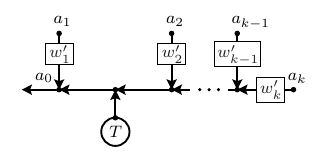}\\
    \downarrow\\
    \dessin{2.5cm}{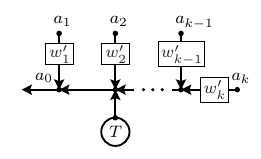}
  \end{array}
\raisebox{-.75cm}{\rotatebox{30}{$\rightarrow$}}\hspace{.5cm} \dessin{2.5cm}{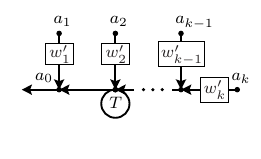}.
  \]
  The first operation is an expansion  in the ``k=1 case'';  the second step is a leading
contraction, and the final operation is a contraction using the induction hypothesis. 
Finally, we undo all the previous steps to get back to the
initial configuration with the desired co-leading edge contracted.

Invariance under a change of ordering of the edges at some vertex is
now easily handled using push, expansion and contraction moves.
First, we note that, up to expansion and contraction moves, the graph can be assumed to be uni-trivalent. Then the procedure
  is schematically illustrated as follows:
\[
 \dessin{2.15cm}{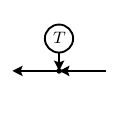}\leftrightarrow \dessin{2.15cm}{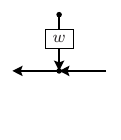}\leftrightarrow \dessin{2.15cm}{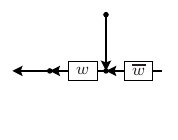}\leftrightarrow \dessin{2.15cm}{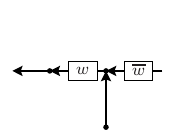}\leftrightarrow \dessin{2.15cm}{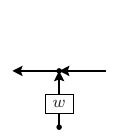}\leftrightarrow \dessin{2.15cm}{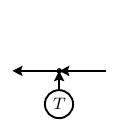}
  \]
There, $T$ denotes a tree that we would like to ``move down'' by reversing the cyclic order at the trivalent vertex. 
With contraction and push moves, we first replace $T$ by a linear graph $T_L$: this is schematically denoted by a $w$-labeled edge in the above figure.  
Next we push down all the decorations of $T_L$, so that the vertical edges are empty and can be contracted. 
The reversed procedure can then be performed with the cyclic order reversed. 

We next check the invariance under an R1 move. 
At the level of the images by $\overline\xi_{SL}$, such a move creates two or one arrow(s), depending on whether the edge where the move is performed, is an extra edge or not. 
The head $h$ and tail $t$ of such an arrow are separated by tails and sequences of heads such that the concatenation of the corresponding words is trivial as a free group element. Lemma \ref{lem:Star} can hence be applied to pull $t$ next to $h$. The  arrow can then be removed using a Reid1 move.
 
 For R3 moves, consider in some welded graph $\Gamma$, an edge between two vertices $a$ and
 $b$, and decorated by a word $w$, and suppose that $\Gamma$ contains some other edge decorated by $wb$. 
 Let $\Gamma'$ be the result of the R3 move that replaces the latter label by $wa^w=aw$.
By construction, the linear welded graph $\Gamma_l$ associated with $\Gamma$ (step 2 of the definition of $\overline\xi_{SL}(\Gamma)$) also contains a $w$--labeled edge between vertices $a$ and $b$. 
 In the welded string link $\overline\xi_{SL}(\Gamma)$, this yields bunches of tails corresponding, respectively, to all occurrences of $a$ and $b$ in the
 edge decorations, separated by a sequence of heads which correspond to the word
 $w$. Pulling one tail from the $b$--bunch through the $w$--sequence
 of heads will conjugate the associated head by $w$. The resulting welded string link is precisely $\overline\xi_{SL}(\Gamma')$.

This concludes the proof that the map $\overline\xi_{SL}$ is  well-defined.
Moreover, it clearly satisfies
$\overline\xi_{SL}\circ\overline\psi_{SL}=\textrm{Id}$. This shows that
$\overline\psi_{SL}$ is injective, hence bijective, and the proof is complete in the string link case. 
\saut

The case of welded links is handled in the same way, except that extra edges are contracted in order to get a
cyclic $w$--graph instead of a linear one. 
There is however one extra issue: the resulting cyclic $w$--graph has no canonical orientation. 
This affects the 3-step procedure defining $\overline \xi_L$, since in step 3 we need to choose an arbitrary orientation on each circle component; 
this is precisely why one needs to introduce the Global Reversal move in the link case, which identifies two Gauss diagrams that differ by this choice of orientation. A more subtle point is that reversing the orientation on some component also affects step 1, as follows: 
\[
 \dessin{2.5cm}{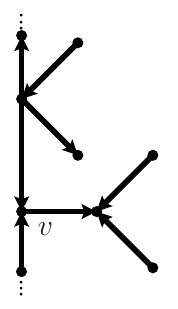}\quad  \longrightarrow \quad \dessin{2.5cm}{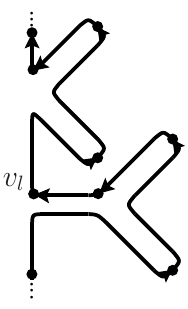}\quad \textrm{ or }\quad \dessin{2.5cm}{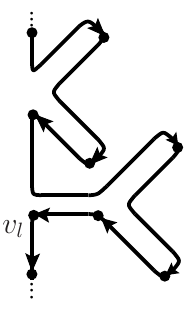}.
\]
As the figure illustrates, in the linearization algorithm given in
step 1, the chosen orientation is used at each vertex $v$ of a
$w$--graph, to define the corresponding vertex $v_l$ in the linear
$w$--graph. This vertex $v_l$ represents, in the resulting Gauss
diagrams, a bunch of arrow tails:  Lemma \ref{lem:Star} ensures that
the Gauss diagrams resulting from the above two situations, are actually equivalent.
This concludes the proof of Theorem \ref{th:Embed}.
\end{proof}

\section{Extra moves on welded graphs}
\label{sec:ExtraMoves}

In this section, we explain how the moves of Figure \ref{fig:ExtraWeldedMoves} on welded graphs, follow from the C, OR, S and R3 moves of Figure \ref{fig:WeldedMoves}, as claimed at the end of Section \ref{sec:WeldedGraphs}.
We also show how one version the $\Upsilon$ move (in the sense of Remark \ref{rem:versions}) is implied by usual welded moves.

\subsection{Push and split moves}\label{sec:ExtraMoves0}
We first observe that the extra push (P) and split (Split) moves are consequences of the C, OR and S moves.   
For this purpose, we first note that, up to OR moves, the S move can be generalized to the case where some edge is oriented toward the vertex  where the move is performed; in that case, the prefix is then reversed and added at the end of the corresponding decoration:
     \[
     \dessin{2.06cm}{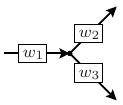}\xleftrightarrow{\ \textrm{OR,S,OR}\ }
     \dessin{2.06cm}{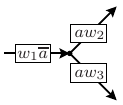}.
 \]
 Applying this variant of the S move iteratively, provides a P move: 
     \[
     \dessin{2.06cm}{CGM6t1}\xleftrightarrow{\ \textrm{S,$\dots$,S}\ }
     \dessin{2.06cm}{CGM6t2}.
   \]
 Moreover, combining this variant of the S move with a C move, allows to ``split''  edge decorations, thus providing a Split move:
   \[
 \dessin{1.24cm}{Split_1}\xleftrightarrow{\ \textrm{E}\ }
 \dessin{1.24cm}{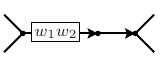}\xleftrightarrow{\ \textrm{S}\ } \dessin{1.24cm}{Split_3}.
 \]
\subsection{Rephrased R3 and generalized stabilization move}\label{sec:ExtraMoves1}
Next, we show that the Split, C, S and R3 moves on welded graphs, 
realize a rephrased R3 move. 
More precisely, we illustrate below how some edge label $w_1w_2$ can be replaced by $w_1 \overline w\,\overline awb w_2$, where $(a,b,w)$ is some other edge in the graph: 
   \[
   \begin{array}{ccccc}
 \dessin{1.24cm}{Insert_1} &\!\xleftrightarrow{\textrm{\ Split,E\ }}\!&
 \dessin{1.24cm}{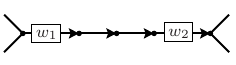} & & \\
  &\!\xleftrightarrow{\textrm{\ S\ }}\!&\dessin{1.24cm}{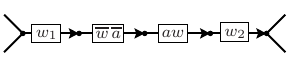}& & \\
     &\!\xleftrightarrow{\textrm{\ R3\ }}\!&
                           \dessin{1.24cm}{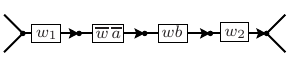}
     &\!\xleftrightarrow{\textrm{\ Split\ }}\!&\dessin{1.24cm}{Insert_6}.
   \end{array}
   \]
The other versions of this move, that insert 
any cyclic permutation of $\overline w\,\overline awb$, or any cyclic
   permutation of it, are proved in a similar way.\\
   
Finally, the figure below illustrates how the generalized stabilization move, is achieved by a sequence of C,S and R3 moves: 
   \[
   \begin{array}{ccc}
     \dessin{2.75cm}{Stab_1}
     &\!\xleftrightarrow{\textrm{\ E,S\ }}\!&
      \dessin{2.75cm}{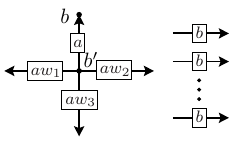}\\
     &\!\xleftrightarrow{\textrm{\ R3,$\dots$,R3}}\!&   \dessin{2.75cm}{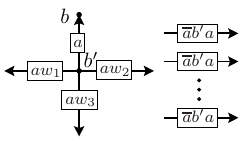}\\
     &\!\xleftrightarrow{\textrm{\ S\ }}\! & \dessin{2.75cm}{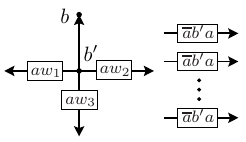}\\
                 &\!\xleftrightarrow{\textrm{\ C\ }}\!&    \dessin{2.75cm}{Stab_6}
   \end{array}
 \]

\subsection{A trivial  $\Upsilon$ move}\label{sec:ExtraMoves2}

We finally show that, in the special case where $w=1$, $\eta=-\e$ and involving only the four self-arrows shown in the figure
of Definition \ref{def:HulaHoop}, the $\Upsilon$ move follows from Reid1,2,3 moves. 
This is illustrated in the following figure: 
 \[
   \begin{array}{ccccccc}
   \dessin{3.5cm}{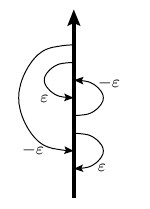}
    &\xrightarrow{\textrm{Reid1}}&
   \dessin{3.5cm}{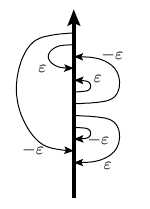}
   & \xrightarrow{\textrm{Reid3}}&
    \dessin{3.5cm}{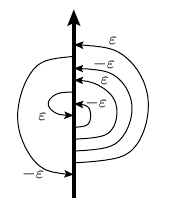}
      & \raisebox{-.5cm}{\rotatebox{325}{$\xrightarrow{\textrm{Reid2}}$}}&
                                      \\[-1cm]
&&&&&&\dessin{3.5cm}{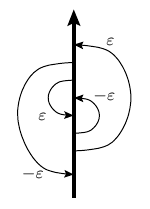}\\[-1cm]
        \dessin{3.5cm}{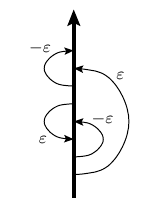}
    &\xleftarrow{\textrm{Reid1}}&
   \dessin{3.5cm}{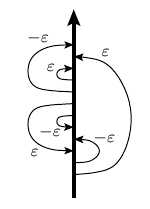}
   & \xleftarrow{\textrm{Reid3}}&
    \dessin{3.5cm}{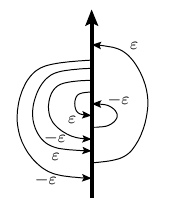}
      & \raisebox{.5cm}{\rotatebox{35}{$\xleftarrow{\textrm{Reid2}}$}}&\\
   \end{array}
 \]
 The strategy is to connect the two sides of the move to a common
 symmetric configuration. For both, the deformation starts by inserting two (self)
 arrows with Reid1 moves, then combines two Reid3 moves with TC moves
 to rearrange the relative positions of the arrow tails, and deletes a pair of arrows with a Reid2 move.
 \\
To the authors' knowledge, this is the only version of the $\Upsilon$ move that follows from the usual welded moves. 

\bibliographystyle{abbrv}
\bibliography{References}

\end{document}